\documentclass[smallextended]{svjour3}       
\smartqed
\RequirePackage[OT1]{fontenc}
\RequirePackage[numbers]{natbib}
\RequirePackage[colorlinks,citecolor=blue,urlcolor=blue]{hyperref}
\usepackage{ae,amssymb, amsmath, natbib, graphicx, euscript, mathrsfs, enumerate, empheq,color, lscape, nicefrac, tabularx}
\usepackage{caption}
\usepackage{cmap}
\usepackage{fancyhdr}
\usepackage{pgfplots}
\pgfplotsset{compat=newest}
\usepackage[tikz]{bclogo}
\usepackage{algorithm2e}

\usetikzlibrary{shapes, calc, matrix, positioning, arrows,snakes,backgrounds, patterns, trees}
\newcommand{\1}{\mathbb{I}}

\renewcommand{\ell}{L}

\def\R{\mathbb{R}}

\newcommand{\eps}{\varepsilon}

\newcommand{\lb}{\left\{}
\newcommand{\rb}{\right\}}

\renewcommand{\P}{{\mathbb P}}

\renewcommand{\kappa}{\varkappa}

\usepackage{multirow}
\usepackage{arydshln}

\begin{document}
\title{Extreme value analysis for mixture models with heavy-tailed impurity 
\thanks{ 
The first author is supported by the RSF grant No 17-11-01098.  }}

\titlerunning{Extreme value analysis for mixture models with heavy-tailed impurity }

\author{Vladimir Panov $ (corresponding \; author)$\\Ekaterina Morozova
}

\authorrunning{V.~Panov and E.~Morozova} 

\institute{HSE University\\
Laboratory of Stochastic Analysis and its Applications\\
             Pokrovsky boulevard 11, 109028 Moscow, Russia\newline
             \\                         
	\email{vpanov@hse.ru         \and  morekaterina@gmail.com
	}
}

\date{Received: \today}
\maketitle

\begin{abstract}
This paper deals with the extreme value analysis for the triangular arrays, which appear when some parameters of the mixture model vary as the number of observations grow. When the mixing parameter is small, it is natural to associate one of the components with "an impurity" (in case of regularly varying distribution, "heavy-tailed impurity"), which "pollutes" another component.  
We show that 
 the set of possible limit distributions is much more diverse than in the classical Fisher-Tippett-Gnedenko theorem, and provide the numerical examples showing the efficiency of the proposed model for studying the maximal values of  the stock returns. 
 
\end{abstract}

\setcounter{tocdepth}{3}
\setcounter{secnumdepth}{3}

\section{Introduction}
Consider the mixture model
\begin{eqnarray}\label{main}
F(x; \eps, \vec{\theta}) = (1-\eps) F^{(1)}(x; \vec\theta^{(1)}) + \eps F^{(2)}(x, \vec{\theta}^{(2)}),
\end{eqnarray} 
where \(\eps \in (0,1) \) is a mixture parameter, \(F^{(1)}(x; \vec{\theta}^{(1)})\) and \(F^{(2)}(x; \vec{\theta}^{(2)})\) are CDFs of two distributions parametrised by vectors \(\vec{\theta}^{(1)}, \vec{\theta}^{(2)}\) correspondingly, and \(\vec{\theta}=(\vec{\theta}^{(1)} , \vec{\theta}^{(2)})\). In this paper we focus on the case when the second component in this mixture corresponds to some heavy-tailed distribution, while the first one can be either light- or heavy-tailed.   When \(\eps\) is small, the second component can be referred to as \textit{the heavy-tailed impurity.}\footnote{The term "heavy-tailed impurity" is known in the context of percolation theory, see~\cite{vanBN}. Here we use it in more general set-up, following~\cite{gm}.} Some applications of this approach are described by Grabchak and Molchanov (\citeyear{gm}). For instance,  in population dynamics,  this approach can be used for modelling the migration of species: the distance of migration of most species can be modelled by light-tailed distribution, but there is a small number of species with  "very active" behaviour. 

It would be a worth mentioning that the parameters \(\eps\) and \(\vec\theta\)  may depend on the number of available observations, denoted below by \(n\). For instance, in the aforementioned example from population dynamics, the proportion of "very active" species decays when the total number of species grows.  In this context, the distribution of resulting variable changes with \(n\), and this model can be considered as the infinitesimal triangular array   --- a collection of real random variables \(\{ X_{nj}, j=1..k_{n}\},\) \(k_{n} \to \infty\) as \(n \to \infty,\)  such that \(X_{n1}, ..., X_{nk_{n}}\) are independent for each \(n\). The classical limit theorems for this class of models are well-known in the literature, see, e.g., monographs by Petrov (\citeyear{Petrov}), Meerschaert and Scheffler (\citeyear{ms}). For instance, it is known that the class of possible non-degenerate limit laws of the sums \(X_{n1}+...+X_{n k_{n}} -c_{n}\) with deterministic \(c_n\) and triangular array \(\{ X_{nj}, j=1..k_{n}\}\)  satisfying the assumption of infinite smallness,\begin{eqnarray*}
\label{inf}
\forall\; \delta>0: \qquad 
\sup_{j=1..k_{n}}
\P \left\{
	\left|
		X_{nj}
	\right| 
	> \delta
\right\} \to 0, \qquad n \to \infty,
\end{eqnarray*}
 coincides with the class of infinitely-divisible distributions.

Surprisingly, there are very few  papers dealing with the extreme value analysis for this model. To the best of our knowledge, there exists no general statements describing the class of non-degenerate limits of 
\begin{eqnarray}
\label{lim}
(\max_{j=1..k_n} X_{nj} - c_n)/s_n, \qquad n \to \infty, 
\end{eqnarray}
with deterministic \(c_n, s_n\). Clearly, the convergence to types theorem  is  applicable to this situation, and guarantees that the limit law is determined up to the change of location and scale. Nevertheless, unlike the well-known Fisher-Tippett-Gnedenko theorem, the class of limit distributions in~\eqref{lim} includes not only the  Gumbel, Fr{\'e}chet and Weibull laws.  Some conditions guaranteeing the convergence of the triangular array to some limit are given by Freitas and H{\"usler} (\citeyear{FH}), but their result  essentially employ the assumption that the limit distribution is twice differentiable, which is violated in the examples of the model~\eqref{main} provided below.
 Let us mention here that other known papers on this topic are concentrated on some particular examples yielding convergence to the Gumbel law, see Anderson, Coles and H{\"u}sler (\citeyear{ACH}),  Dkengne, Eckert and Naveau (\citeyear{DEN}).

In the first part of the paper (Section~\ref{WRV}) we consider the particular case of~\eqref{main}, when the first  component has the Weibull distribution (and therefore it is in the maximum domain of attraction of the Gumbel law), while the heavy-tailed impurity is modelled by the regularly varying  distribution (MDA of the Fr{\'e}chet law). Note that in the classical setting, when the parameters \(\eps\) and \(\vec{\theta}\) are fixed, the limit behaviour of the sum is determined by the second  component, and therefore the maximum under proper normalisation converges to the Fr{\'e}chet law.  Interestingly enough, even in the case, when only one parameter (namely the mixing parameter \(\eps\)) varies, the set of possible limit distributions includes Gumbel and Fr{\'e}chet distributions  and also one discontinuous law.  The exact statement is formulated in Theorem~\ref{thm1}.

In Section~\ref{WTRV} we turn towards  more complicated model, which appears when one uses the truncated regularly varying distribution for the second component, and the truncation level \(M\) grows with \(n.\) This part of our research is motivated by a discussion  concerning the choice between truncated and non-truncated Pareto-type distributions, see Beirlant, Alves and Gomes (\citeyear{BIG}). The asymptotic behaviour depends on the rate of growth of \(M\): as we show, the resulting conditions are related to the soft, hard and intermediate truncation regimes introduced by  Chakrabarty and Samorodnitsky (\citeyear{CS}). 
 Note that in that paper it is shown that the softly truncated regularly varying distribution has heavy tails (understood in the sense of the non-Gaussian limit law for the sum), and therefore the term "heavy-tailed impurity" can be also used for  models of this kind. 

The main theoretical contribution of our research is formulated as Theorem~\ref{thm2},  dealing with the case when both \(\eps\) and \(M\) depend on \(n.\) It turns out that the set of possible limit laws in~\eqref{lim} includes 6 various distributions, and, for some sets of parameters, maximal value diverges under any (also nonlinear) normalisation.  Our theoretical findings are illustrated by the simulation study (Section~\ref{ss}). 

The choice of the  Weibull distribution for the first component in~\eqref{main} is partially based on the great popularity of this distribution in applications, see, e.g.,  the overview by Laherrere and Sornette (\citeyear{LS}). As we  show in  Section~\ref{BMW}, our model with heavy-tailed impurity is more appropriate for modelling the stock returns as a "pure" model. In this context, our paper continues the discussion started in the paper by Malevergne, Pisarenko and Sornette (\citeyear{MPS}), where it is shown that the tails of the empirical distribution of log-returns decay slower than the tails of the Weibull distribution but faster than the power law.

\section{Weibull-RV mixture}
\label{WRV}
In this section, we  focus on a particular case of the model~\eqref{main}, namely
\begin{equation}
\label{mixture}
F(x; \eps, \vec{\theta}) = (1-\eps)F_1 (x; \lambda, \tau) + \eps F_2(x;  \alpha) 
\end{equation}
where \(\vec{\theta}=(\lambda,\tau, \alpha)\), \(F_1\) is the distribution function of the Weibull law,
\begin{equation}\label{Weib}
F_1 (x; \lambda, \tau)  = F_1(x)= 1-e^{-\lambda x^{\tau}}, \quad x\geq 0, \lambda>0, \tau>0,
\end{equation}
and \(F_2\) corresponds to the regularly varying distribution on \([m,\infty)\),
\begin{equation}\label{RV}
F_2  (x; \alpha) = F_2(x)= 1 - x^{-\alpha} \ell(x), \quad \alpha>0, \quad x\in [m, \infty), 
\end{equation}
with $m = \inf\left\{ x > 0: F_2 (x) >0\right\}$ and a continuous slowly varying function $\ell(\cdot)$. Let us recall that  by definition,
\begin{eqnarray*}
\lim_{x \to \infty} \ell(tx)/\ell(x) = 1, \qquad \forall t >0, 
\end{eqnarray*}
and the term "slow variation" comes from the property 
\begin{eqnarray}\label{prop1}
x^{-\epsilon} L(x) \to 0 \quad\mbox{ and }\quad 
x^{\epsilon} L(x) \to \infty \qquad \mbox{as } x \to \infty
\end{eqnarray}
for every \(\epsilon>0.\) The extensive overview of the properties of slowly varying functions is given in \cite{Bingham} and \cite{Resnick}.

As we already mentioned in the introduction,  the first component  is in the MDA of the Gumbel law, while the second is in the MDA of the Fr{\'e}chet law. In Appendix~\ref{app1}, we show that the the mixture distribution function \(F\) is in the MDA of the Fr{\'e}chet law, provided that  the parameters \(\eps\) and \(\vec{\theta}\) are fixed.


In what follows we consider the  case when the mixing parameter \(\eps=\eps_n\) decays to zero as \(n\) grows. It is natural  to slightly generalise the model to the form of row-wise independent triangular array
\begin{eqnarray}
\label{triang}
X_{nj} \sim F(x; \eps_n, \vec{\theta}), \qquad n\geq 1,\qquad  j = 1 .. k_n, \end{eqnarray}  where   \(k_n\) is an unbounded increasing sequence, and for any \(i=1..n\),  the r.v.'s $X_{nj}, j=1..k_n$ are independent. The set-up allowing various numbers of elements in different rows is standard both in studying the classical limit laws (see~\cite{Petrov}) and in the extreme value theory (see~\cite{DEN}).

As we show in the next theorem,  the asymptotic behaviour of the maximum in this model is determined by the rate of growth of \(k_n\), the rate of decay of \(\eps_n\) and the slowly varying function \(L.\)  Note that the rates of \(\log k_n\)  and \(k_n \eps_n\) are compared in terms of the following three alternative conditions,
\[
\tag{A1}\label{A1}
\exists \beta>\nicefrac{\tau}{\alpha}: \qquad
\lim\limits_{n\to\infty} \frac{\log k_n}{(k_n \eps_n)^{\beta}} = \infty,\]
\[
\tag{A2}\label{A2}
\exists \beta \in (0,\nicefrac{\tau}{\alpha}): \qquad 
\lim\limits_{n\to\infty} \frac{\log k_n}{(k_n \eps_n)^{\beta}} = 0,
\]
\[\tag{A3}
\label{A3}
\exists{c>0}: \qquad
k_n \eps_n = c(\log k_n)^{\nicefrac{\alpha}{\tau}}.
\]

  
\begin{figure}[h]
\begin{tikzpicture}[sibling distance=3cm, punkt/.style={rectangle, draw=black}, block/.style={rectangle, draw=black}]
\node {$k_n \eps_n$}
child { node[punkt, rectangle split, rectangle split parts=2]{$k_n \eps_n \xrightarrow[n\to\infty] {}0$
\nodepart{second} Gumbel}}
child { node {$k_n \eps_n \xrightarrow[n\to\infty] {} \infty$}
child {[sibling distance=3.6cm] child{ node[punkt, rectangle split, rectangle split parts=2] {$\exists \beta\in(0,\nicefrac{\tau}{\alpha}): \lim\limits_{n\to\infty} \frac{\log k_n}{(k_n \eps_n)^{\beta}} = 0$
         \nodepart{second} Fr\'echet}}
child{node{$k_n \eps_n = c(\log k_n)^{\nicefrac{\alpha}{\tau}}$}
child{child{node[punkt, rectangle split, rectangle split parts=2]{$\ell(u)\xrightarrow[u\to\infty]{}\infty$
\nodepart{second} Fr\'echet}}
child{child{node[punkt, rectangle split, rectangle split parts=2]{$\ell(u)\xrightarrow[u\to\infty]{}\tilde{c}>0$
\nodepart{second} $\begin{cases}0,& x< \lambda^{-\nicefrac{1}{\tau}} (c\tilde{c})^{-\nicefrac{1}{\alpha}},\\
e^{-(1+x^{-\alpha})},& x=\lambda^{-\nicefrac{1}{\tau}} (c\tilde{c})^{-\nicefrac{1}{\alpha}},\\
e^{-x^{-\alpha}},& x>\lambda^{-\nicefrac{1}{\tau}} (c\tilde{c})^{-\nicefrac{1}{\alpha}}
\end{cases}$}}}
child{node[punkt, rectangle split, rectangle split parts=2]{$\ell(u)\xrightarrow[u\to\infty]{}0$
\nodepart{second} Gumbel}}}}
child {node[punkt, rectangle split, rectangle split parts=2]{$\exists \beta>\nicefrac{\tau}{\alpha}: \lim\limits_{n\to\infty} \frac{\log k_n}{(k_n \eps_n)^{\beta}} = \infty$
\nodepart{second} Gumbel}}}}
child { node[punkt, rectangle split, rectangle split parts=2] {$k_n \eps_n = const$
\nodepart{second} Gumbel}};
\end{tikzpicture}
\caption{Possible limit distributions for maxima of the triangular array~\eqref{mixture}}
\label{fig:thm_1}
\end{figure}
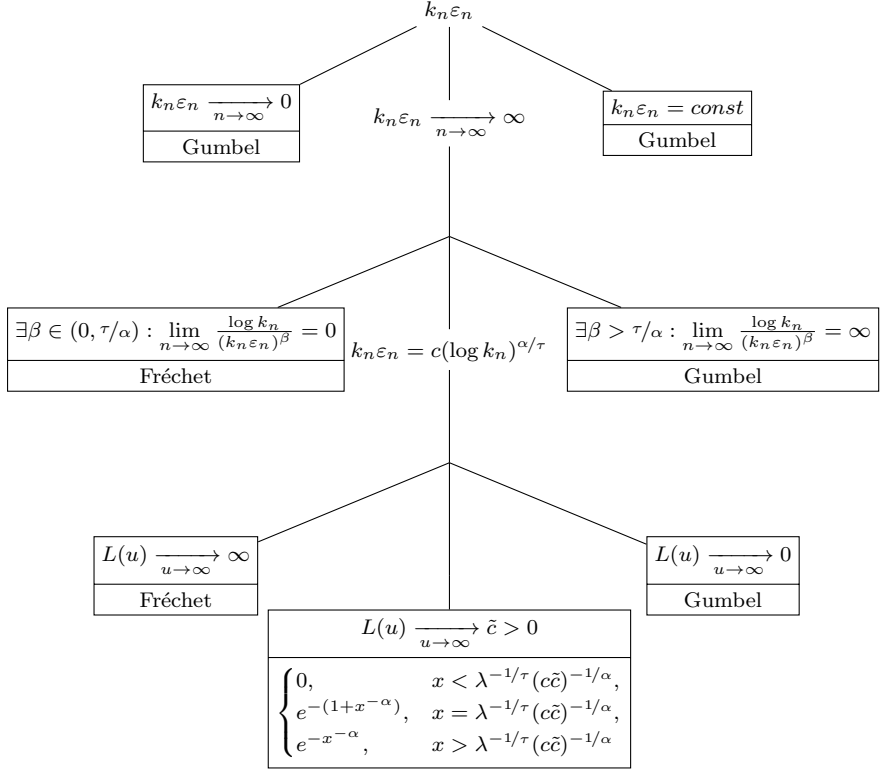

\begin{theorem}
\label{thm1}
Consider the row-wise independent triangular array~\eqref{triang}. Assume that 
\(\lim_{x \to \infty} L(x) \in [0, \infty].\)\footnote{This assumption means that 
the slowly varying function \(L(x)\) doesn't exhibit infinite oscillation. The counterexample to this condition is given  in~\cite{MikoschRegular}, Example~1.1.6.}
 Then for any sequences \(k_n\to\infty\), $\eps_n\to 0$  there exist deterministic  sequences \(c_n, s_n\) such that 
\begin{equation}\label{state}
\lim\limits_{n\to\infty} \P\lb\max\limits_{j=1,\ldots, k_n} X_{nj} \leq s_n x + c_n\rb = H(x), \quad \forall x\in\R
\end{equation}
with some non-degenerate limit law \(H(x).\) More precisely, \(H(x)\) belongs to the type of the following three distribution functions\footnote{Due to the convergence to types theorem (Theorem A1.5 from~\cite{EKM}), if \(H(x)\) is the distribution function of the limit law in~\eqref{state}, then any other non-degenerate law appearing in~\eqref{state} under another normalisation is of the form \(H(ax+b)\) with some constants \(a,b\).}:

\begin{enumerate}[(i)]
\item \underline{Gumbel distribution,} \(H(x)=e^{-e^{-x}}\), if and only if any of the following conditions is satisfied 
\begin{enumerate}
\item $k_n\eps_n\to 0$ as $n\to\infty$ or $k_n \eps_n = const>0$;
\item $k_n\eps_n\to \infty$ as $n\to\infty$, and~\eqref{A1} holds;
\item \eqref{A3} holds, and $\ell(u)\to 0$ as $u\to\infty$.
\end{enumerate}
In all cases, possible choice of the normalising sequences is  
\begin{eqnarray}\label{const1}
 s_n = (\lambda\tau)^{-1} (\lambda^{-1} \log k_n)^{\nicefrac{1}{\tau} - 1}\quad\mbox{ and } \quad c_n = (\lambda^{-1} \log k_n)^{\nicefrac{1}{\tau}}.\end{eqnarray}
\item \underline{Fr{\'e}chet distribution} with parameter \(\alpha\), \(H(x)=e^{-x^{-\alpha}}\), if and only if any of the following conditions is satisfied 
\begin{enumerate}
\item \eqref{A2} holds;
\item \eqref{A3} holds, and \(\ell(u) \to \infty\) as \(u \to \infty\).
\end{enumerate}
In all cases, one can take 
\begin{eqnarray}\label{const2}
s_n = F_2 ^{\leftarrow} (1-(k_n \eps_n)^{-1}), \qquad \mbox{ and} \qquad c_n = 0,
\end{eqnarray}
where \(F_2 ^{\leftarrow}(y)=\inf\{x \in \R: F_2(x) \geq y\}\) for \(y \in [0,1].\)
\item \underline{Discontinuous distribution} with cdf 
\[
H(x)=
\begin{cases}
0,& x< \lambda^{-\nicefrac{1}{\tau}} (c\tilde{c})^{-\nicefrac{1}{\alpha}},\\
e^{-(1+x^{-\alpha})},& x=\lambda^{-\nicefrac{1}{\tau}} (c\tilde{c})^{-\nicefrac{1}{\alpha}},\\
e^{-x^{-\alpha}},& x>\lambda^{-\nicefrac{1}{\tau}} (c\tilde{c})^{-\nicefrac{1}{\alpha}},
\end{cases}\]
if  \eqref{A3} holds, and $\ell(u)\to\tilde{c}$ for some $\tilde{c}>0$ as $u\to\infty$. The normalising sequences can be fixed in the form~\eqref{const2}.
\end{enumerate}
\end{theorem}
\begin{proof}
The proof is given in Appendix~\ref{app2}.
\end{proof}

The graphical representation of this result is presented in Figure~\ref{fig:thm_1}.

\section{Weibull-truncated RV mixture}\label{WTRV}
Now we consider one more complicated model, such that  the distribution of the second component in~\eqref{mixture} also changes as \(n\) grows. 
Consider the mixture distribution
\begin{equation}
\label{mixture2}
F(x; \eps, M, \vec{\theta}) = (1-\eps)F_1 (x; \lambda, \tau) + \eps \widetilde{F}_2(x; \alpha,  M),
\end{equation}
where as before, \(\vec{\theta}=(\lambda,\tau, \alpha)\), \(F_1\) is the distribution function of the Weibull law (see~\eqref{Weib}), while \(\widetilde{F}_2\) is the upper-truncated regularly varying distribution,
\begin{equation}
\label{truncated RV}
\widetilde{F}_2(x; \alpha, M) = 
\begin{cases}
\frac{F_2 (x;  \alpha)}{F_2 (M; \alpha)},& \text{ if } x\in [m, M],\\
1,& \text{ if } x>M
\end{cases}
\end{equation}
with  $F_2 (x;  \alpha)$ corresponding to a regularly varying distribution~\eqref{RV}. 

It would be an interesting mentioning that  the components in this model correspond to different maximum domains of attraction: the maximum for the first component under proper normalisation converges to the Gumbel law, while the second --- to the Weibull law, see Appendix~\ref{app3}.  

By analogue with~\eqref{triang}, we consider the triangular array 
\begin{eqnarray}
\label{triang2}
X_{nj} \sim F(x; \eps_n, M_n, \vec{\theta}), \qquad n\geq 1,\qquad  j = 1 .. k_n, \end{eqnarray}  where   \(k_n, M_n\) are unbounded increasing sequences, and for any \(i=1..n\),  the r.v.'s $X_{nj}, j=1..k_n$ are independent.  Note that the classical limit laws for this model (law of large numbers  and limit theorems for the sums)  are essentially established in \cite{Panov2017}.

The next theorem reveals the asymptotic behaviour of the maximal value depending on the rates of \(\eps_n, M_n, k_n\), and the properties of the slowly varying function \(L.\) An important difference from the model considered in Section~\ref{WRV}  is that in some cases  the limit distribution is degenerate for any (also non-linear) normalising sequence.

It turns out that if  \(k_n \eps_n\)  tends to any finite constant, then the limit distribution is Gumbel. Our findings in the remaining case \(k_n \eps_n \to \infty\) are presented in Table~\ref{tab:thm2}. The asymptotic behaviour of the maximum is determined by the asymptotic properties of the sequences \(k_n, \eps_n\) in terms of \eqref{A1}-\eqref{A3}, and the rate of growth of  \(M_n\) in terms of the following three alternating conditions:
\[\tag{M1} \label{M1}
\exists \gamma>\nicefrac{1}{\alpha}: \qquad
\lim\limits_{n\to\infty} \frac{M_n}{(k_n \eps_n)^{\gamma}} =\infty,
\]
\[\tag{M2}
\label{M2}
\exists \gamma\in(0,\nicefrac{1}{\alpha}): 
\qquad \lim\limits_{n\to\infty} \frac{M_n}{(k_n \eps_n)^{\gamma}}=0; \]
\[\tag{M3}
\label{M3} \exists \breve{c}>0: \qquad
M_n = \breve{c}(k_n \eps_n)^{\nicefrac{1}{\alpha}}.\]
The conditions \eqref{M1}-\eqref{M3} are  related to the notion of hard- and soft truncation. Following~\cite{CS}, we say that that a variable \(\eta\) is truncated softly, if 
\begin{eqnarray}\label{soft}
\lim_{n \to \infty} k_n \P\bigl\{ | \eta | \geq M_n \bigr\} = 0.
\end{eqnarray}
For a regularly varying distribution of \(\eta\), the condition~\eqref{soft} holds if there exists \(\gamma>1/\alpha\) such that \(M_n/k_n^\gamma \to \infty\). This fact follows from 
\begin{eqnarray*}
 k_n \P\bigl\{ | \eta | \geq M_n \bigr\}  = k_n M_n^{-\alpha} L(M_n) \lesssim k_n M_n^{-\alpha+ \epsilon}=(M_n/k_n^{1/(\alpha-\epsilon)})^{(-\alpha+\epsilon)}\end{eqnarray*}
for any $\epsilon>0$\footnote{Here and below we mean by $f_n \gtrsim g_n$ that $\lim\limits_{n\to\infty}( f_n/ g_n)=\infty$.}
. Analogously, \(\eta\) is truncated hard, that is,
 \begin{eqnarray*}
\lim_{n \to \infty} k_n \P\bigl\{ | \eta | \geq M_n \bigr\} = \infty,
\end{eqnarray*}
if there exists \(\gamma \in (0,1/\alpha)\) such that \(M_n/k_n^\gamma \to 0.\) 

Our results for the case \eqref{M1} (see first raw in Table~\ref{tab:thm2}) coincide with the findings from \cite{CS}:  \textit{in the soft truncation regime, truncated power tails behave, in important respects, as if no truncation took place.}  In fact, in our setup, the results are completely the same as for the non-truncated distribution considered in Theorem~\ref{thm1}. 

Our outcomes for \eqref{M2} (second raw in Table~\ref{tab:thm2}) are quite close to another finding from \cite{CS}, namely, 
\textit{in the hard truncation regime much of "heavy tailedness" is lost}. Actually, we get that the behaviour is determined by the first component except the case~\eqref{A2} with \(\lim_{n \to \infty} k_n e^{-\lambda M_n ^{\tau}} \ne \infty\).

Finally, the intermediate case~\eqref{M3} (third raw in Table~\ref{tab:thm2}) is divided into various subcases. The comparison with \cite{CS} is not possible because the authors decide \textit{to largely leave} this question \textit{aside in this article, in order to keep its size manageable}. In our research, we provide the complete study of this case.

The exact result is formulated below. 
\begin{table}[t]
\caption{Possible limit distributions for maxima of the triangular array~\eqref{triang2}}
\label{tab:thm2}
\begin{tabular}{c|c|c|c}
\hline
 \(k_n \eps_n \underset{n\to\infty}{\longrightarrow}\infty\)& \eqref{A1} & \eqref{A2} & \eqref{A3} \\ \hline
 \multirow{6}{*}{\eqref{M1}} & \multirow{6}{*}{Gumbel} & \multirow{6}{*}{Fr\'echet} & Gumbel, \\ 
 & & &  if \(L(u)\underset{u\to\infty}{\longrightarrow} 0\) \\ \cdashline{4-4}[1pt/1pt]
 & & &  Fr\'echet, \\ 
 & & & if \(L(u)\underset{u\to\infty}{\longrightarrow} \infty\) \\ \cdashline{4-4}[1pt/1pt]
 & & & Distribution~I, \\
 & & & if \(L(u) \underset{u\to\infty}{\longrightarrow}\tilde{c} \in (0, \infty)\)\\ \hline
 \multirow{4}{*}{\eqref{M2}} & \multirow{4}{*}{Gumbel} & Gumbel,  & \multirow{4}{*}{Gumbel} \\
 & & if \(k_n \gtrsim e^{\lambda M_n ^{\tau}}\) &\\ \cdashline{3-3}[1pt/1pt]
 & & no limit, & \\
 & & if \(k_n \gtrsim e^{\lambda M_n ^{\tau}}\) is not fulfilled \\ \hline
 \multirow{15}{*}{\eqref{M3}} & \multirow{15}{*}{Gumbel} & \multirow{2}{*}{Fr\'echet,}  & Gumbel, \\ 
 & & & if \(L(u)\underset{u\to\infty}{\longrightarrow} 0\) \\ \cdashline{4-4}[1pt/1pt]
 & & \multirow{3}{*}{if \(L(u)\underset{u\to\infty}{\longrightarrow} 0\)} & Gumbel, \\ 
 & & & if \(L(u)\underset{u\to\infty}{\longrightarrow}\tilde{c}\in(0, \infty]\), \\ 
 & & & and \(\lambda ^{\nicefrac{1}{\tau}}\breve{c} c^{\nicefrac{1}{\alpha}}\in(0, 1)\) \\\cdashline{3-4}[1pt/1pt]
 & & \multirow{3}{*}{Distribution~II, } & Distribution~III, \\
 & & & if \(L(u)\underset{u\to\infty}{\longrightarrow}\tilde{c}\in(0, \infty)\), \\ 
 & & & and \(\lambda ^{\nicefrac{1}{\tau}}\breve{c} c^{\nicefrac{1}{\alpha}}>1\) \\ \cdashline{4-4}[1pt/1pt]
 & & \multirow{3}{*}{if \(L(u) \underset{u\to\infty}{\longrightarrow} \tilde{c}\in(0, \infty)\)} & Distribution~IV,  \\
 & & & if \(L(u)\underset{u\to\infty}{\longrightarrow}\tilde{c}\in(0, \infty)\),\\
 & & & and \(\lambda ^{\nicefrac{1}{\tau}}\breve{c} c^{\nicefrac{1}{\alpha}}=1\) \\ \cdashline{3-4}[1pt/1pt]
 & & no limit, & no limit, \\
 & & \multirow{2}{*}{if \(L(u)\underset{u\to\infty}{\longrightarrow} \infty\)} & if \(L(u)\underset{u\to\infty}{\longrightarrow} \infty\), \\ 
 & & & and \(\lambda ^{\nicefrac{1}{\tau}}\breve{c} c^{\nicefrac{1}{\alpha}}\geq 1\) \\ \hline
\end{tabular}
\end{table}
\begin{theorem}
\label{thm2}
Consider the row-wise independent triangular array~\eqref{triang2} under the assumption that  \(\eps_n \to 0, M_n \to \infty, k_n \to \infty\) as \(n \to \infty.\) Assume also that 
\(\lim_{x \to \infty} L(x) \in [0, \infty].\) 

Then the non-degenerate limit law \(H(x)\) for the properly normalised row-wise maximum \(\max\limits_{j=1,\ldots, k_n} X_{nj} \) (see~\eqref{state}) belongs to the type of the following  distributions. 

\begin{enumerate}[1.]
\item \underline{Gumbel distribution,} \(H(x)=e^{-e^{-x}}\), if and only if any of the following conditions  is satisfied \begin{enumerate}
\item[1.1] $k_n\eps_n\to 0$ as $n\to\infty$ or $k_n \eps_n = const>0;$
\item[1.2] $k_n\eps_n\to \infty$ as $n\to\infty$, and moreover 
\begin{itemize}
\item \eqref{M1} and \eqref{A1} hold;
\item  \eqref{M1} and \eqref{A3} hold, and \(L(u) \to 0\) as \(u \to \infty;\)
\item \eqref{M2} and \eqref{A1} hold;
\item \eqref{M2} and \eqref{A2} hold, and \(k_n \gtrsim e^{\lambda M_n ^{\tau}};\)
\item \eqref{M2} and \eqref{A3} hold; 
\item \eqref{M3} and \eqref{A1} hold;
\item \eqref{M3} and \eqref{A3} hold, and \(L(u) \to 0\) as \(u \to \infty;\)
\item \eqref{M3} and \eqref{A3} hold,  \(L(u) \to \tilde{c} \in (0,\infty]\) as \(u \to \infty,\) and $\lambda^{\nicefrac{1}{\tau}}\breve{c} c^{\nicefrac{1}{\alpha}}\in(0, 1)$.
\end{itemize}
\end{enumerate}
In all cases, possible choice of the normalising sequences is  given by~\eqref{const1}. 
\item \underline{Fr{\'e}chet distribution} with parameter \(\alpha\), \(H(x)=e^{-x^{-\alpha}}\), if and only if any of the following conditions is satisfied 
\begin{itemize}
\item \eqref{M1} and \eqref{A2} hold;
\item \eqref{M1} and \eqref{A3} hold, and \(L(u) \to \infty\) as \(u \to \infty;\)
\item \eqref{M3} and \eqref{A2} hold, and \(L(u) \to 0\) as \(u \to \infty.\)
 \end{itemize}
In all cases, one can take \(s_n, c_n\) in the form~\eqref{const2}.
\item \underline{Special cases}:
\begin{enumerate}
\item[(I)] \begin{eqnarray*}
H(x)=
\begin{cases}
0,& x< \lambda^{-\nicefrac{1}{\tau}} (c\tilde{c})^{-\nicefrac{1}{\alpha}},\\
e^{-(1+x^{-\alpha})},& x=\lambda^{-\nicefrac{1}{\tau}} (c\tilde{c})^{-\nicefrac{1}{\alpha}},\\
e^{-x^{-\alpha}},& x>\lambda^{-\nicefrac{1}{\tau}} (c\tilde{c})^{-\nicefrac{1}{\alpha}},
\end{cases}\end{eqnarray*}
provided \eqref{M1} and \eqref{A3} hold, and \(L(u) \to \tilde{c} \in (0, \infty)\);
\item[(II)] 
\[
H(x) = \begin{cases} e^{\tilde{c}\breve{c}^{-\alpha} - x^{-\alpha}},& x\in(0, \breve{c}\tilde{c}^{-\nicefrac{1}{\alpha}}],\\ 1,& x>\breve{c}\tilde{c}^{-\nicefrac{1}{\alpha}}, \end{cases}
\]
provided \eqref{M3} and \eqref{A2} hold, $\ell(u)\to\tilde{c} \in (0,\infty)$;
\item[(III)]
\[
H(x)=
 \begin{cases}
0,& x<\lambda^{-\nicefrac{1}{\tau}}(c\tilde{c})^{-\nicefrac{1}{\alpha}},\\
\exp\lb -1+\tilde{c}\breve{c}^{-\alpha} - \lambda^{\nicefrac{\alpha}{\tau}} c\tilde{c} \rb,& x=\lambda^{-\nicefrac{1}{\tau}}(c\tilde{c})^{-\nicefrac{1}{\alpha}},\\
\exp\lb \tilde{c}\breve{c}^{-\alpha}-x^{-\alpha} \rb,& x\in(\lambda^{-\nicefrac{1}{\tau}}(c\tilde{c})^{-\nicefrac{1}{\alpha}}, \breve{c}\tilde{c}^{-\nicefrac{1}{\alpha}}],\\
1,& x> \breve{c}\tilde{c}^{-\nicefrac{1}{\alpha}},\end{cases}\]
provided \eqref{M3} and \eqref{A3} hold,  \(L(u) \to \tilde{c} \in (0,\infty)\) as \(u \to \infty,\) and $\lambda^{\nicefrac{1}{\tau}}\breve{c} c^{\nicefrac{1}{\alpha}}>1$;
\item[(IV)]  
\[
H(x) = \begin{cases}
0,& x<\breve{c}\tilde{c}^{-\nicefrac{1}{\alpha}},\\
e^{-1},& x=\breve{c}\tilde{c}^{-\nicefrac{1}{\alpha}},\\
1,& x> \breve{c}\tilde{c}^{-\nicefrac{1}{\alpha}},\end{cases}
\]
provided \eqref{M3} and \eqref{A3} hold,  \(L(u) \to \tilde{c} \in (0,\infty)\) as \(u \to \infty,\) and $\lambda^{\nicefrac{1}{\tau}}\breve{c} c^{\nicefrac{1}{\alpha}}=1$.
\end{enumerate}
In all cases the normalising sequences can be chosen as in~\eqref{const2}.
\end{enumerate}
 \underline{The limit distribution is degenerate for any sequences \(s_n\) and \(c_n\)} in the following three cases:
\begin{itemize}
\item \eqref{M2} and \eqref{A2} hold, and the condition \(k_n \gtrsim e^{\lambda M_n ^{\tau}}\) is not fulfilled; 
\item \eqref{M3} and \eqref{A2} hold, and \(L(u) \to \infty\) as \(u \to \infty\);
\item \eqref{M3} and \eqref{A3} hold,  \(L(u) \to \infty\) as \(u \to \infty,\) and $\lambda^{\nicefrac{1}{\tau}}\breve{c} c^{\nicefrac{1}{\alpha}} \geq 1$.
\end{itemize} 
Moreover, in these cases the distribution of 
\begin{eqnarray*}
\lim\limits_{n\to\infty} \P\lb\max\limits_{j=1,\ldots, k_n} X_{nj} \leq v_n(x) \rb 
\end{eqnarray*}
is degenerate for any increasing sequence \(v_n(x),\) which is unbounded in \(n\) and \(x.\)

\end{theorem}
\begin{proof}
The proof is given in Appendix~\ref{app4}. 
\end{proof}
\section{Simulation study} \label{ss}

\begin{table}[t]
\caption{The values of \(\eps_n\) and  \(M_n\) chosen for the numerical study}
\label{simulation}
\setlength{\tabcolsep}{28pt}
\begin{tabular}{c|c|c}
\hline
 & \eqref{A1} & \eqref{A2}\\ \hline
\multirow{2}{*}{\eqref{M1}} & \(\eps_n = n^{-1} \log n, \) & \(\eps_n = n^{-1} (\log n)^2, \) \\ 
& \(M_n = \log (n+1)\) & \(M_n = (\log (n+1))^2\) \\ \hline
\multirow{2}{*}{\eqref{M2}} & \(\eps_n = n^{-1} \log n, \) & \(\eps_n = n^{-1} (\log n)^2,\) \\ 
& \(M_n = \sqrt{\log (n+1)}\) & \(M_n = \sqrt{\log (n+1)}\) \\\hline
\end{tabular}
\end{table}

\begin{figure}[t]
	\centering
	\includegraphics[width=1\linewidth]{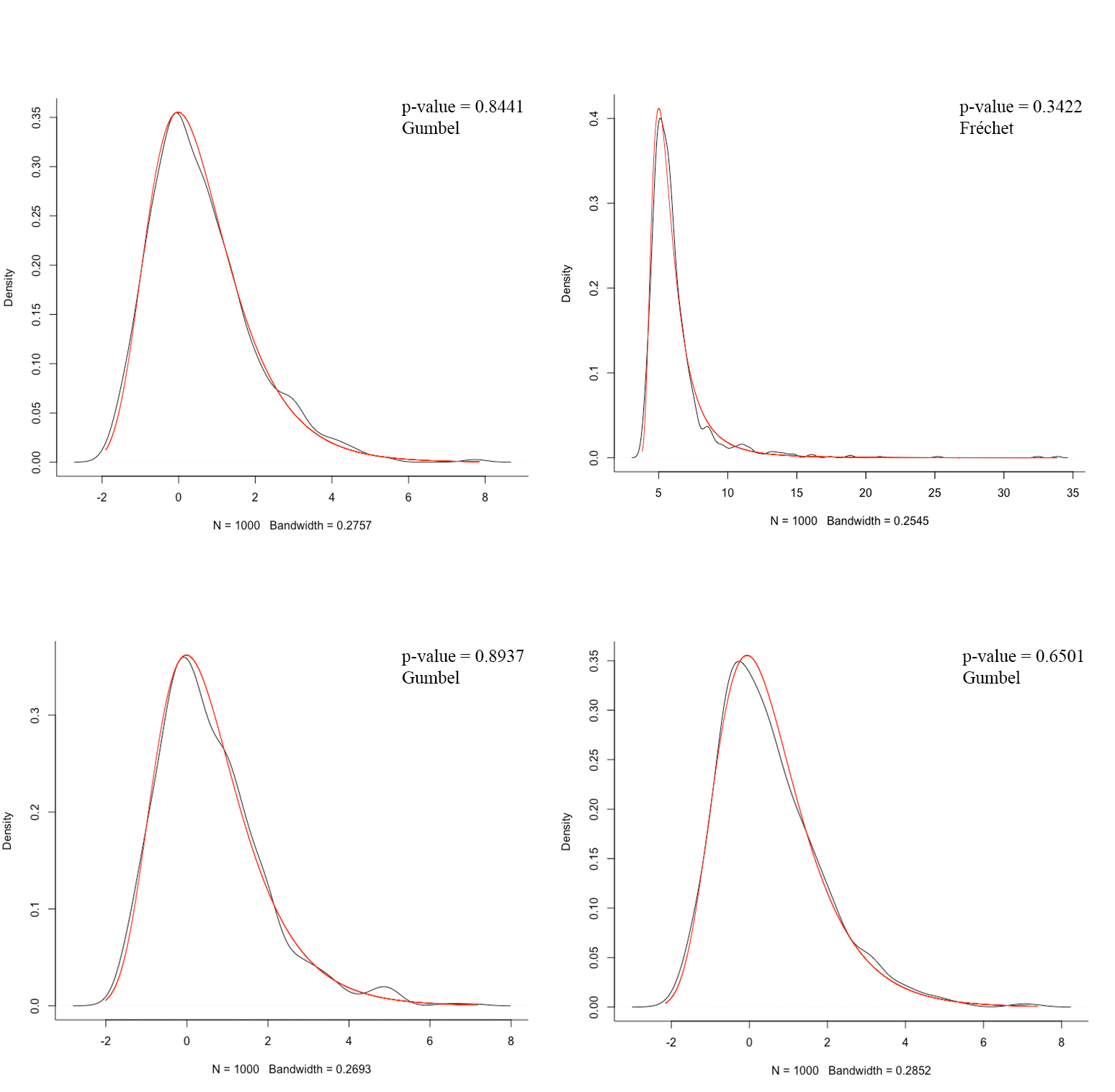}
	\caption{Densities of sample maxima for groups \eqref{A1}-\eqref{M1} (top left), \eqref{A2}-\eqref{M1} (top right), \eqref{A1}-\eqref{M2} (bottom left) and \eqref{A2}-\eqref{M2} (bottom right) superimposed with the theoretical limit distributions suggested by Theorem~\ref{thm2}
	}
	\label{fig:simstudy}
\end{figure}

The aim of the current section is to illustrate the dependence of limit distribution for maxima in the model~\eqref{triang2} on the rates of the mixing parameter \(\eps_n\) and of the truncation level \(M_n\). For this purpose we consider four triangular arrays~\eqref{triang2} with \(k_n = n\) having all permanent  parameters the same, namely, \(\vec{\theta}=(1, 1, 1.5, 1)\) and \(m=0.1\). The sequences \(\eps_n\) and \(M_n\) are chosen to satisfy the following pairs of conditions: \eqref{A1}-\eqref{M1}, \eqref{A2}-\eqref{M1}, \eqref{A1}-\eqref{M2} and \eqref{A2}-\eqref{M2}. The exact form of mixing and truncation parameters are presented in Table~\ref{simulation}. 

As previously, the primary separation is made due to the rates of \(\log k_n\) and \(k_n \eps_n\): we fix \(\eps_n = n^{-1} (\log n)\) and \(\eps_n = n^{-1} (\log n)^2\), which imply conditions~\eqref{A1} and~\eqref{A2}, respectively. Next, the models are divided according to the rate of growth of \(M_n\) in the form \(M_n = (\log (n+1))^a\) with \(a=1/2,1,2.\)  Recall that from Theorem~\ref{thm2}, it follows that the limit distribution is Gumbel for the pairs  \eqref{A1}-\eqref{M1}, \eqref{A1}-\eqref{M2} and \eqref{A2}-\eqref{M2} (note that for the last two cases \(k_n \gtrsim e^{\lambda M_n ^{\tau}}\) under our choice), and Fr{\'e}chet for the pair \eqref{A2}-\eqref{M1}.

For each case we simulate 1000 samples of length 1000, and find the maximal value of each sample. The goodness-of-fit of the limit distributions of the maximal values suggested by Theorem~\ref{thm2} is  tested by the Kolmogorov-Smirnov criterion.   Figure~\ref{fig:simstudy} depicts the kernel density estimates of the densities of normalised maxima in each case superimposed with the limit distributions implied by Theorem~\ref{thm2}. It can be seen that for all groups the density estimates are  quite close to the theoretical densities, and the Kolmogorov-Smirnov test does not reject the null of the corresponding theoretical distributions (corresponding p-values are given on the same figure). 

\section{Modelling the log-returns of BMW shares}
\label{BMW}
\begin{figure}[t]
	\centering
	\includegraphics[width=1\linewidth]{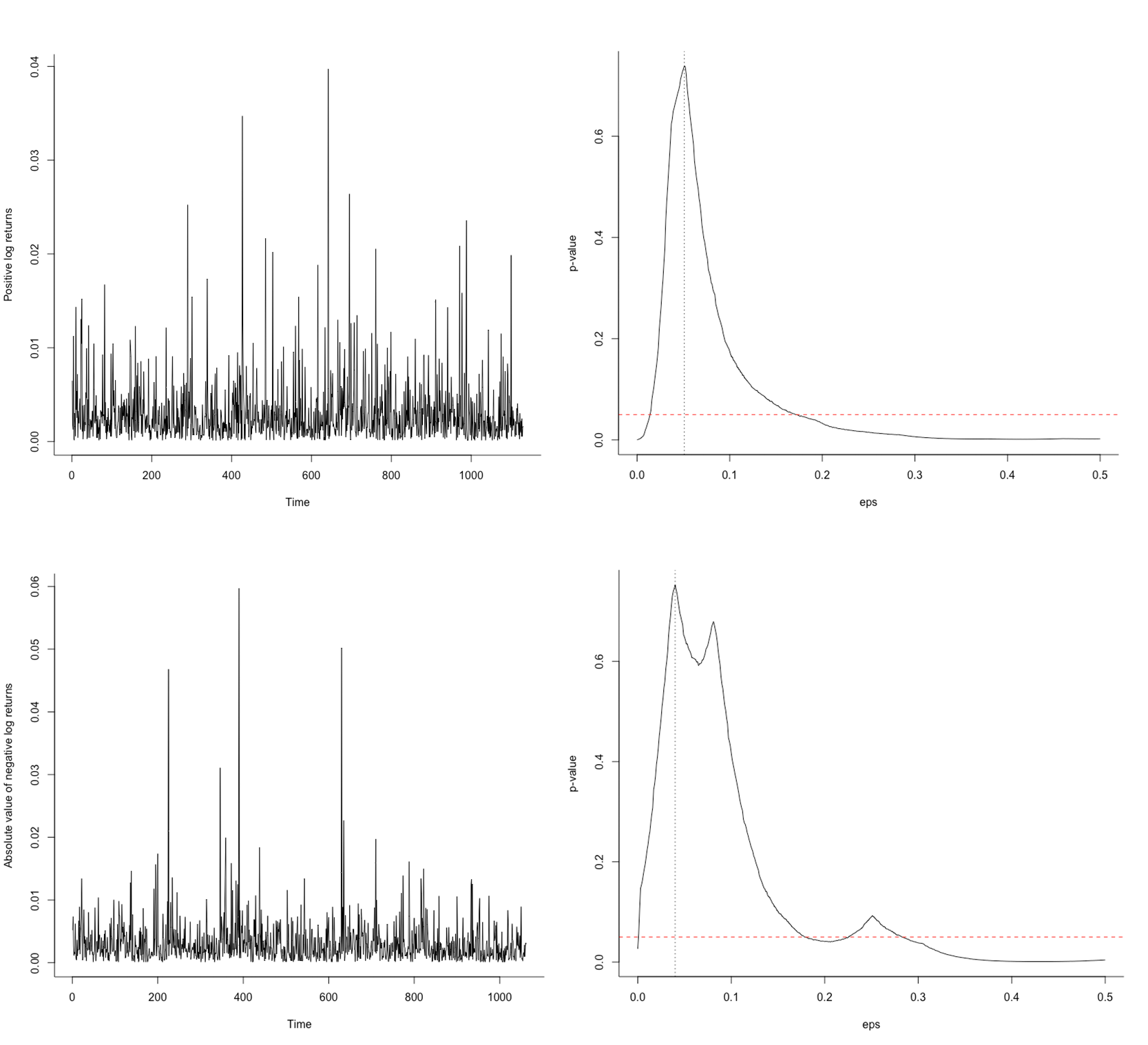}
	\caption{First (second) row: the plot of data; p-values for Weibull fit for positive (absolute negative) log returns
	}
	\label{fig:returns}
\end{figure}

Starting from the prominent paper by Mandelbrot \cite{Man}, heavy-tailedness of distributions of price changes is a well-known stylised fact, leading to the frequent choice of power laws for the modelling, see, e.g., \cite{Cont}. However, numerous papers admit that the tails of the distributions used for modelling the returns is though heavier than normal, yet lighter than of a power law. For instance, Laherr\`ere and Sornette  \cite{LS} demonstrate that daily price variations on the exchange market can be successfully described by the Weibull distribution with parameter smaller than one. Malevergne et al. \cite{MPS} intently analyse financial returns on different time scales, ranging from daily to 5- and 1-minute data, and come to the conclusion that the Pareto distribution fits the highest 5\% of the data, while the remaining 95\% are most efficiently described by the Weibull law. Thus, it is reasonable to expect that the overall distribution of returns should be successfully described by the model which (in some sense) lies in between of these two distributions. This idea serves as a motivation of the application of  the model~\eqref{mixture} to modelling the log-returns.

In our study we consider hourly logarithmic returns of BMW shares in 2019. 
Following~ \cite{MPS}, we analyse positive and negative returns separately.  The sample sizes are equal to 1130 and 1062, respectively. The plots for positive and negative log returns are presented in the first plot in Figure~\ref{fig:returns}. In what follows, we assume that the log-returns are jointly independent. This assumption was checked by  the chi-squared test resulting in p-values 0.234 and 0.223 for positive and negative returns, respectively.




The analysis consists of 4 steps. Below we denote the positive log-returns by \(X_1,..., X_{1130}\) and the negative log-returns by \(Y_1,..., Y_{1062}\).
\begin{enumerate}
\item
\begin{figure}[t]
	\includegraphics[width=1\linewidth]{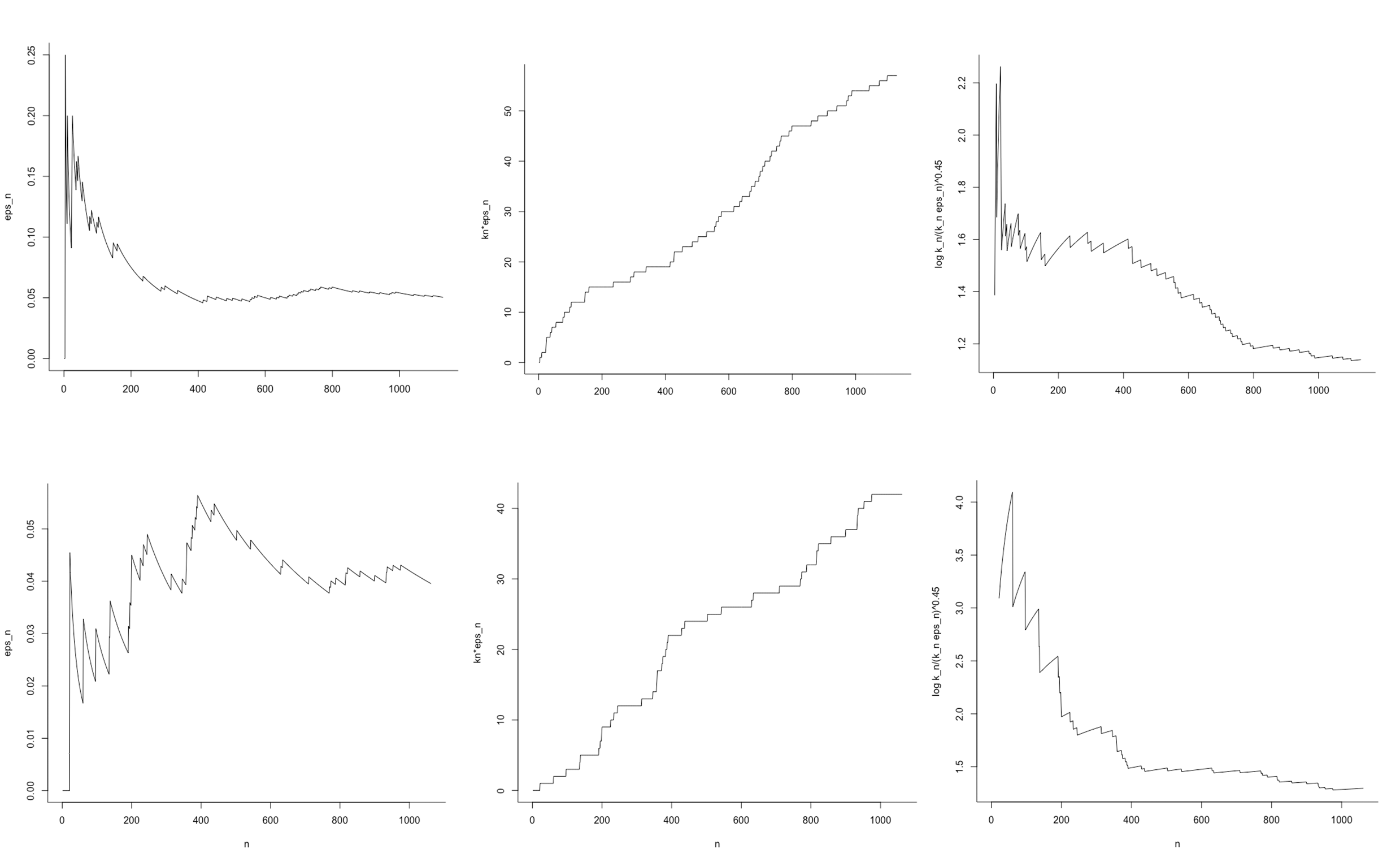}
	\caption{$\hat{\eps}_n$ against $n$ (left), $k_n \hat{\eps}_n$ against $n$ (middle), $\log k_n (k_n \eps_n)^{-0.45}$ against $n$ (right) for positive log returns of BMW (first row) and absolute values of negative log returns of BMW (second row)}
	\label{fig:max_BMW}
\end{figure}
\textbf{Separation of components.} For each \(n=1...1130\), the sample \(X_1,..., X_{n}\) is divided into 2 parts corresponding to the first and the second components in~\eqref{mixture2}, where the slowly varying function \(L\) is equal to a constant. For such partition we assign all observations except the \(\lfloor 1130\cdot\eps_{1130} \rfloor\) greatest order statistics of the whole sample to the first component and test the goodness-of-fit for Weibull distribution by the Kolmogorov-Smirnov criterion. The values of \(\eps_{1130}\) are taken on a grid from 0 to 0.5 with a step of 0.001. From the second plot in Figure~\ref{fig:returns} it can be seen that there is an evident peak in p-values. The corresponding value \(\hat{\eps}_{1130}\) is considered to be an estimate of the mixing parameter for \(n=1130\), and the \(\lfloor 1130\cdot\hat{\eps}_{1130} \rfloor\) upper order statistics are assumed to come from the heavy-tailed part. For all \(n=1...1129\) the parameter \(\eps_n\) is then estimated as the proportion of elements corresponding to the second component. The same procedure is applied for each \(n=1..1062\) to the sample \(Y_1,..., Y_{n}\). 

The results are illustrated by~Figure~\ref{fig:max_BMW}. The first plot in two rows indicates that in both cases $\eps_n$ declines with $n$, though in case of negative log returns the decrease is not so evident. From the second plot one can see that $k_n \eps_n$ with $k_n=n$ appear to tend to infinity. Therefore, as suggested by~Theorems~\ref{thm1} and~\ref{thm2}, we examine the asymptotic behaviour of the ratio $\log k_n/ (k_n \eps_n)^{\beta}$ for different values of $\beta>0$. For both positive and absolute values of negative log returns we get that for $\beta=0.45$ this ratio decreases rapidly. \newline

\item \textbf{Model selection.} Based on the partition obtained on the previous step, the decision between truncated and non-truncated distributions for the second component is made based on the test proposed by Aban~\cite{Aban}. From Figure~\ref{fig:aban} one observes that the null hypothesis of non-truncated law is not rejected for both positive and negative log returns since the p-values are significantly larger than 0.05. Thus, it can be concluded that the model~\eqref{mixture} is more appropriate for the considered data. The Kolmogorov-Smirnov test does not reject the null of Pareto distribution for the observations assigned to the second component with p-values 0.971 and 0.925 for positive and negative log returns, respectively. It should be noted that in both cases the Pareto distribution does not fit the whole sample, since the p-values are smaller than \(10^{-16}\). \newline

\begin{figure}[t]
	\centering
	\includegraphics[width=1\linewidth]{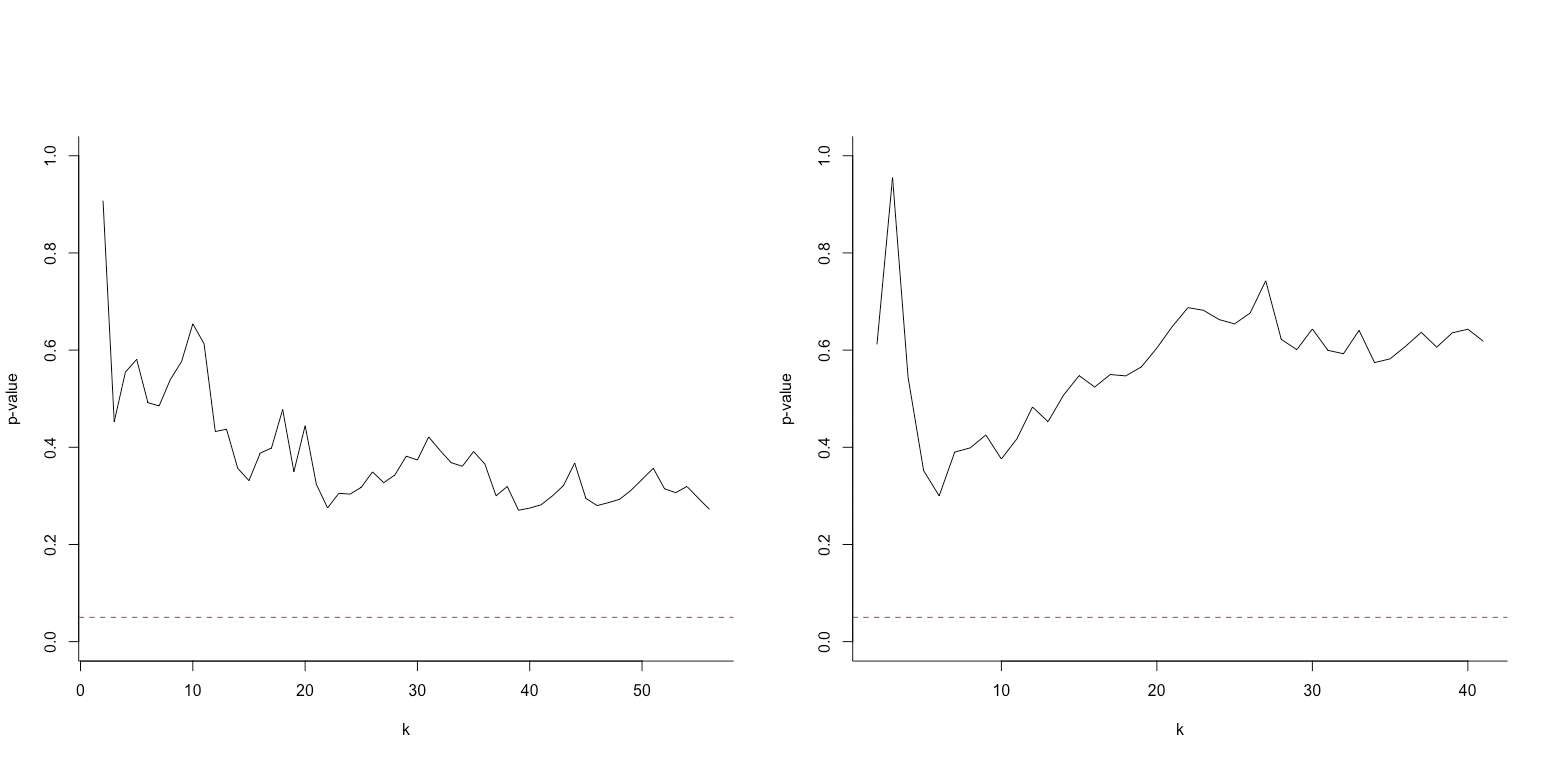}
	\caption{The p-values of the Aban's test for positive (left) and absolute values of negative (right) log returns
	}
	\label{fig:aban}
\end{figure}

\item \textbf{Estimation of parameters.} The parameters of the first and second components are estimated by the maximum-likelihood approach. The estimated values are presented in Table~\ref{est_BMW}. Since $\hat{\tau}/ \alpha$ is equal to 0.483 for positive log returns and 0.484 for absolute values of negative, we conclude that the assumption~\eqref{A2} is fulfilled with \(\beta=0.45\), and therefore the limit distribution for maxima is the Fr\'echet distribution, see item (ii) in Theorem~\ref{thm1}.

 It is worth mentioning that for both positive and absolute values of negative log returns we get $\hat{\alpha}>2$, which is completely coherent with general empirical results for financial returns and addresses the common critique against models with infinite variance, see~\cite{Cont}.\newline

\begin{table}[t]
\caption{Estimated values of the parameters of mixture distribution for positive and absolute values of negative log returns of BMW, 2019}\label{est_BMW}
\setlength{\tabcolsep}{10.5pt}
	\begin{tabular}{c|*{6}{c}}
	\hline\noalign{\smallskip}
		Estimates &&   $\hat{\lambda}$  & $\hat{\tau}$  & $\hat{\alpha}$ & $\hat{m}$ & $\hat{\eps}$   \\ \hline 
		Positive log returns && 0.003 & 1.278 & 2.649 & 0.009 & 0.051   \\ 
		Abs. negative log returns && 0.003 & 1.246 & 2.573 & 0.01 & 0.04 \\
		\noalign{\smallskip}\hline
	\end{tabular}
\end{table}
\begin{figure}[t]
	\includegraphics[width=1\linewidth]{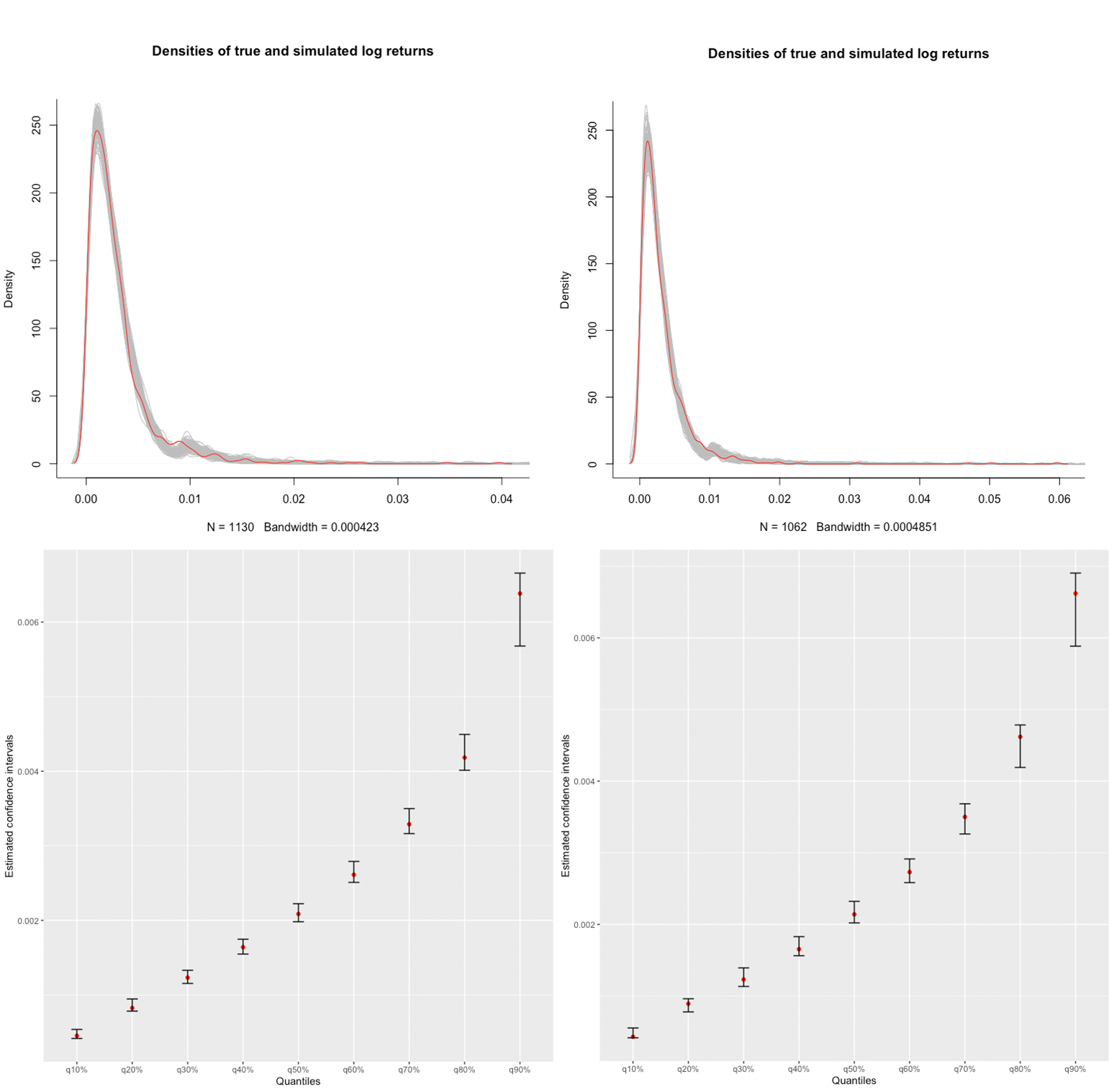}
	\caption{Top left (right): real (red) and simulated (grey) density of positive (negative) log returns; bottom left (right): empirical quantiles of positive (negative) log returns and the corresponding confidence intervals}
	\label{fig:BMWfit}
\end{figure}
\begin{table}[t]
\caption{Empirical quantiles of positive log returns of BMW, 2019, and the estimated confidence intervals}
\label{table_quant1}
\setlength{\tabcolsep}{5.5pt}
	\begin{tabular}{c|*{9}{c}}\hline  
		Quantiles\(\cdot 10^3\) &   10\%  & 20\%  & 30\% & 40\% & 50\% & 60\% & 70\% & 80\% & 90\%   \\ \hline 
		Lower CI & 0.416 & 0.783 & 1.154 & 1.548 & 1.982 & 2.509 & 3.164 & 4.013 & 5.679  \\ 
		Estimate & 0.452 & 0.825 & 1.232 & 1.64 & 2.085 & 2.611 & 3.289 & 4.183 & 6.383 \\ 
		Upper CI & 0.537 & 0.946 & 1.33 & 1.747 & 2.223 & 2.79 & 3.498 & 4.493 & 6.657 \\ \hline
	\end{tabular}
\end{table}
\begin{table}[t]
\caption{Empirical quantiles of absolute values of negative log returns of BMW, 2019, and the estimated confidence intervals}
\label{table_quant2}
\setlength{\tabcolsep}{5.5pt}
	\begin{tabular}{c|*{9}{c}}\hline  
		Quantiles \(\cdot 10^3\) &   10\%  & 20\%  & 30\% & 40\% & 50\% & 60\% & 70\% & 80\% & 90\%   \\ \hline 
		Lower CI & 0.417 & 0.778 & 1.134 & 1.562 & 2.021 & 2.583 & 3.262 & 4.191 & 5.884  \\ 
		Estimate & 0.431 & 0.892 & 1.231 & 1.654 & 2.14 & 2.73 & 3.501 & 4.619 & 6.62 \\ 
		Upper CI & 0.554 & 0.962 & 1.393 & 1.829 & 2.321 & 2.914 & 3.684 & 4.784 & 6.905 \\ \hline
	\end{tabular}
\end{table}
\item \textbf{Validation of the model.}
Figure~\ref{fig:BMWfit} depicts the true density of positive (top left) and absolute values of negative (top right) log returns superimposed with densities of 100 simulations from the mixture~\eqref{mixture} with the corresponding parameter estimates. The constructed model is also verified by the empirical confidence intervals for the sample quantiles based on 100 simulations. The results are given in Tables~\ref{table_quant1} and~\ref{table_quant2} and illustrated by Figure~\ref{fig:BMWfit}. These intervals are reasonably small and contain all  true values of quantiles. Next, from the Kolmogorov-Smirnov test we conclude that the null of distribution~\eqref{mixture} is not rejected with p-values 0.769 and 0.775 for positive and negative log returns, respectively.  Finally, we arrive at the outcome that the model~\eqref{mixture} is appropriate both for positive and absolute values of negative log returns of BMW at the considered time scale.


\end{enumerate}

\section{Conclusion}
This paper contributes to the existing literature in the following  respects.
\begin{enumerate}
\item We model the heavy-tailed impurity via the mixture of distribution with varying parameters and (following the ideas from~\cite{BIG}, \cite{CS}) consider the resulting model as a triangular array. The notion of heavy-tailed impurity is not new, but all previously known probabilistic results are concentrated only on the classical limit laws, see~\cite{gm}, \cite{Panov2017}. In this paper, we establish the limit laws for the maximum in these models. 
\item The paper delivers an example of the triangular array  such that its raw-wise maximum has (under proper normalisation) 6 different distributions, depending on the rates of the varying parameters.  To the best of our knowledge,  all previous articles on the extreme value analysis for  triangular arrays deal with the  convergence to the limit law  with twice differentiable cdf (\cite{ACH}, \cite{DEN}, \cite{FH}), while some limit distributions in our example are discontinuous.
\item We show the difference between various types of truncation for the regularly varying distributions used for modelling the impurity. Our conditions \eqref{M1}-\eqref{M2} are close to soft and hard  truncation regimes introduced in~\cite{CS}, leading to similar (but not completely the same) outcomes for our mixture model as for the model considered in \cite{CS}. Moreover, unlike previous papers, we study in details the case of intermediate truncation regime~\eqref{M3}.
\item For practical purposes we describe the four-step scheme for the application of this model to the asset price modelling.  This approach can be considered as a possible  development of the idea that the distribution 
of stock returns is in some sense \textit{between exponential and power law}. The comprehensive discussion of this idea can be found in~\cite{MPS}.
\end{enumerate}

\appendix
\section{Classical EVA for the mixture model} \label{app1}
Let us analyse the asymptotic behaviour of maxima of a sequence of i.i.d.\ random variables $X_1, X_2, \dots, X_n$, $n\geq 2$, with cumulative distribution function~\eqref{mixture}. That is, we consider
\begin{equation*}
\lim\limits_{n\to\infty} \P\lb \max\limits_{i=1,\ldots, n} X_i \leq v_n (x) \rb,
\end{equation*}
where $v_n (x)=s_n x +c_n$ is some non-decreasing normalising sequence unbounded in $n$ and $x$. Since \(X_i, i=1..n\) are independent,
\begin{eqnarray*}
\lim\limits_{n\to\infty} \P \lb  \max\limits_{i=1,\ldots, n} X_i \leq v_n (x)\rb &=& \lim\limits_{n\to\infty} \left(\P \lb X_1 \leq v_n (x)\rb\right)^n \\ &=& \lim\limits_{n\to\infty} \left((1-\eps)(1-e^{-\lambda v_n ^{\tau} (x)})+\eps(1-v_n ^{-\alpha} (x) \ell(v_n (x)))\right)^n \\ &=& \exp\lb -\lim\limits_{n\to\infty} n \left((1-\eps)e^{-\lambda v_n ^{\tau} (x)}+\eps v_n ^{-\alpha} (x) \ell(v_n (x))\right)\rb \\ &=& \exp\lb-\lim\limits_{n\to\infty} n\eps v_n ^{-\alpha} (x) \ell(v_n (x)) \right.\\&&\hspace{2cm}\left. \times\left(\frac{1-\eps}{\eps} \cdot\frac{v_n ^{\alpha} (x)}{e^{\lambda (v_n(x))^{\tau}} \ell(v_n (x))} + 1\right)\rb.
\end{eqnarray*}
Since $v_n (x)\to\infty$ as $n\to\infty$, $\lim\limits_{n\to\infty} v_n ^{\alpha} (x) \left(e^{\lambda (v_n(x))^{\tau}} \ell(v_n (x))\right)^{-1} =0$, and therefore
\begin{equation*}
\lim\limits_{n\to\infty} \P \lb  \max\limits_{i=1,\ldots, n} X_i \leq v_n (x)\rb = \exp\lb -\lim\limits_{n\to\infty} n\eps v_n ^{-\alpha} (x) \ell(v_n (x)) \rb.
\end{equation*}
Thus, the limit distribution for maxima is determined by the second component, leading to the Fr\'echet limit. In fact, choosing
\begin{eqnarray*}
c_n=0, \qquad s_n = (1/\bar{F_2})^{\leftarrow} (n), 
\end{eqnarray*}
we get 
 \begin{equation*}
\lim\limits_{n\to\infty} \P \lb  \max\limits_{i=1,\ldots, n} X_i \leq v_n (x)\rb = \exp\lb -\eps x^{-\alpha} \rb,
\end{equation*}
that is the  Fr{\'e}chet-type distribution.
\section{Proof of Theorem~\ref{thm1}}
\label{app2}
For given sequences \(s_n, c_n\), the left-hand side of~\eqref{state} can be represented as 
\begin{eqnarray}
\nonumber
\lim\limits_{n\to\infty} \left(\P\lb X_1 \leq v_n (x) \rb\right)^{k_n}
&=&\lim\limits_{n\to\infty} \left(\bigl(1-\eps_n\bigr)\bigl(1-e^{-\lambda v_n ^{\tau} (x)}\bigr) + \eps_n \bigl(1-v_n ^{-\alpha} (x) \ell(v_n (x))\bigr)\right)^{k_n}\\
\nonumber
&=&\lim\limits_{n\to\infty} (1 - e^{-\lambda v_n ^{\tau} (x)} - \eps_n v_n ^{-\alpha} (x) \ell(v_n (x)))^{k_n}\\
\label{limm}
&=& \exp\Bigl\{-\lim_{n \to \infty} \Bigl( 
k_n  e^{-\lambda v_n ^{\tau} (x)} + k_n \eps_n v_n ^{-\alpha} (x) \ell(v_n (x))
\Bigr) \Bigr\},
\end{eqnarray}
where \(v_n(x) = s_n x +c_n\). Our aim is to find the sequences \(s_n, c_n\) guarantying that this limit (denoted by \(H(x)\)) is non-degenerate. We divide the range of possible rates of convergence of \(k_n, \eps_n\) into several essentially different cases.
\begin{itemize}
\item[(\textit{i})] Let $k_n\eps_n\to 0$ as $n\to\infty$ or $k_n \eps_n = const>0$. As $v_n ^{-\alpha} (x) \ell(v_n (x))\to 0$ as $n\to\infty$ by the slow variation of $\ell(\cdot)$ (see~\eqref{prop1}), we get 
\begin{eqnarray*}
H(x) = \exp\Bigl\{-\lim_{n \to \infty} 
 k_n \eps_n v_n ^{-\alpha} (x) \ell(v_n (x))
 \Bigr\},
\end{eqnarray*}
and therefore we deal with the extreme value analysis of the Weibull law. Since $\bar{F} _1 (x) = e^{-\lambda x^{\tau}}, \lambda >0, \tau\geq 1$ is a von Mises function, i.e., $\bar{F} _1 (x)$ can be represented as
\begin{equation*}
\bar{F} _1 (x) = \breve{c}\cdot\exp\lb-\int\limits_{0} ^x \frac{1}{a(u)} \,du\rb, \quad 0 < x < \infty, 
\end{equation*}
with $\breve{c}=1$ and $a(u) = (\lambda\tau)^{-1} u^{1-\tau}, x>0$,  we get that the limit distribution is Gumbel under the choice  $v_n (x) = s_n x + c_n$ with
\begin{equation*}
c_n := F_1 ^{\leftarrow} \left(1 - \frac{1}{k_n}\right) = \left(\frac{\log k_n}{\lambda}\right)^{\nicefrac{1}{\tau}}
\end{equation*}
\begin{equation*}
s_n := a(c_n) = (\lambda\tau)^{-1} \left(\frac{\log k_n}{\lambda}\right)^{{\nicefrac{1}{\tau}-1}}.
\end{equation*}

\item[(\textit{ii})] Let $k_n\eps_n\to \infty$ as $n\to\infty$. This case is divided into several subcases, depending on the relation between \(s_n\) and \(c_n.\)
\item[1.] First, let us consider $s_n \gtrsim |c_n|$. Then 
\begin{equation*}
v_n (x) = s_n x + c_n = s_n x\left(1+\frac{c_n}{s_n x}\right) = s_n x (1+\bar{o}(1))
\end{equation*}
as $n\to\infty$. Therefore, 
\begin{equation*}H(x)= \exp\Bigl\{-\lim_{n \to \infty} \Bigl(k_n e^{-\lambda s_n ^{\tau} x^{\tau}} + k_n \eps_n s_n ^{-\alpha} x^{-\alpha} \ell (s_n)\Bigr)\Bigr\},
\end{equation*}
as $\ell(s_n x)\sim\ell(s_n)$ for all fixed $x\in\R$ as $n\to\infty$. Clearly, since $c_n$ is not present in the above limit, one can take $c_n = 0$. As for $s_n$, we have \begin{equation}
\label{sn1}
s_n ^{-\alpha} \ell(s_n) = \frac{1}{k_n \eps_n},
\end{equation}
i.e.,
\( s_n := F_2 ^{\leftarrow} \left(1 - (k_n \eps_n)^{-1}\right).
\)
We have
\begin{equation*}
\lim\limits_{n\to\infty} k_n e^{-\lambda s_n ^{\tau} x^{\tau}} =\lim\limits_{n\to\infty} \exp\lb -s_n ^{\tau} \left(\lambda x^{\tau} - \frac{\log k_n}{s_n ^{\tau}}\right)\rb,
\end{equation*}
and therefore the limit distribution in \eqref{limm} is non-degenerate (and is actually the Fr{\'e}chet distribution) if and only if
\begin{equation}
\label{cond1}
\lim\limits_{n\to\infty} \frac{\log k_n}{s_n ^{\tau}} = 0.
\end{equation}
It would be a worth mentioning that \(s_n\) depends on the function \(L\) via the equality~\eqref{sn1}. Let us recall that $\ell(\cdot)$ is slowly varying and therefore
\begin{equation*}
\ell(s_n) \gtrsim s_n ^{-\epsilon} \quad \forall\epsilon>0.
\end{equation*}
Thus, from~\eqref{sn1} we get
\begin{equation*}
\frac{1}{k_n \eps_n} = s_n ^{-\alpha} \ell(s_n) \gtrsim s_n ^{-\alpha-\epsilon} \quad \forall\epsilon>0,
\end{equation*}
and 
\begin{equation*}
s_n \gtrsim (k_n \eps_n)^{\nicefrac{1}{(\alpha+\epsilon)}} \quad \forall\epsilon>0.
\end{equation*}
Now, since 
\begin{equation*}
\frac{\log k_n}{s_n ^{\tau}} \lesssim \frac{\log k_n}{(k_n \eps_n)^{\nicefrac{\tau}{(\alpha+\epsilon)}}},
\end{equation*}
we get that for the condition~\eqref{cond1} to be fulfilled, it is sufficient that the right-hand side tends to zero as $n\to\infty$, i.e.,
\begin{equation*}
\lim\limits_{n\to\infty} \frac{\log k_n}{(k_n \eps_n)^{\nicefrac{\tau}{(\alpha+\epsilon)}}} = 0,
\end{equation*}
or, equivalently, 
\begin{equation}
\label{cond1fin}
\exists \beta\in(0,\nicefrac{\tau}{\alpha}): \quad \lim\limits_{n\to\infty} \frac{\log k_n}{(k_n \eps_n)^{\beta}} = 0.
\end{equation}
We conclude that the condition~\eqref{cond1fin} yields~\eqref{cond1}, and in this case the limit distribution is Fr\'echet.

\item[2.] Now, let $s_n>0$ and $c_n\in\R$ be such that $s_n \lesssim |c_n|$. In this case
\begin{equation*}
v_n (x) = s_n x + c_n = c_n \left(\frac{s_n}{c_n} x + 1\right) = c_n(1+\bar{o}(1))
\end{equation*}
as $n\to\infty$. Thus, the limit in~\eqref{limm} takes the form
\begin{equation*}
 H(x)= \exp\Bigl\{-\lim_{n \to \infty} \Bigl( k_n e^{-\lambda (s_n x+c_n) ^{\tau} } + k_n \eps_n c_n ^{-\alpha} \ell (c_n )\Bigr) \Bigr\}.
\end{equation*}

Let the norming constants $c_n$ and $s_n$ be chosen in the form~\eqref{const1}. Then \(H(x)\) is the cdf of the Gumbel law if
\begin{equation}
\label{cond2}
\lim\limits_{n\to\infty} \frac{k_n \eps_n \ell\left((\log k_n)^{\nicefrac{1}{\tau}}\right)}{(\log k_n)^{\nicefrac{\alpha}{\tau}}}=0.
\end{equation}
As previously, we would like to replace~\eqref{cond2} with another condition without $\ell(\cdot)$. Once more, we would like to recall that by slow variation of $\ell(\cdot)$
\begin{equation*}
\ell(x) \lesssim x^{\epsilon} \quad \forall \epsilon>0.
\end{equation*}
From this we conclude that
\begin{equation*}
\frac{k_n \eps_n \ell\left((\log k_n)^{\nicefrac{1}{\tau}}\right)}{(\log k_n)^{\nicefrac{\alpha}{\tau}}} \lesssim \frac{k_n \eps_n}{(\log k_n)^{\nicefrac{(\alpha-\epsilon)}{\tau}}}
\end{equation*}
and the fact that the right-hand side tends to zero as $n\to\infty$ will imply~\eqref{cond2}. In other words, we obtain the Gumbel limit if
\begin{equation*}
\lim\limits_{n\to\infty} \frac{k_n \eps_n}{(\log k_n)^{\nicefrac{(\alpha-\epsilon)}{\tau}}} = \lim\limits_{n\to\infty} \left(\frac{(k_n \eps_n)^{\nicefrac{\tau}{(\alpha - \epsilon)}}}{\log k_n}\right)^{\nicefrac{(\alpha-\epsilon)}{\tau}} = 0,
\end{equation*}
or, equivalently, if
\begin{equation}
\label{cond2fin}
\exists \beta>\nicefrac{\tau}{\alpha}: \quad \lim\limits_{n\to\infty}\frac{\log k_n}{(k_n \eps_n)^{\beta}} = \infty.
\end{equation}

\item[3.] The last possible situation is when neither~\eqref{cond1fin}, nor~\eqref{cond2fin} is satisfied. Clearly, this is the case only when $k_n \eps_n = c\cdot (\log k_n)^{\nicefrac{\alpha}{\tau}}$ for some $c>0$. Not surprisingly, it turns out that the final answer now depends on the asymptotic behaviour of $\ell(\cdot)$.

\item[a)] Let us first consider $\ell(\cdot)$ the case $\ell(u)\to\infty$ as $u\to\infty$. Then one can take $c_n = 0$ and find $s_n$ as the solution to the equation
\begin{equation*}
s_n ^{-\alpha} \ell(s_n) = \frac{1}{k_n \eps_n} = \frac{1}{c\cdot(\log k_n)^{\nicefrac{\alpha}{\tau}}}.
\end{equation*}
The limit for the second component in~\eqref{limm} coincides with the corresponding one in item 2(\textit{i}) (and leads to the cdf of the Fr{\'e}chet law), while for the first component we get
\begin{eqnarray}
\label{lim const}
\lim\limits_{n\to\infty} k_n e^{-\lambda s_n ^{\tau} x^{\tau}} &=& \lim\limits_{n\to\infty} k_n \exp\lb-\lambda\left(\frac{s_n ^{-\alpha} \ell(s_n)}{\ell(s_n)}\right)^{-\nicefrac{\tau}{\alpha}} x^{\tau}\rb \nonumber \\ &=& \lim\limits_{n\to\infty} k_n ^{1-\lambda (c \ell(s_n))^{\nicefrac{\tau}{\alpha}} x^{\tau}}.
\end{eqnarray}
The value of the latter limit is zero for all fixed $x>0$ since $\ell(s_n)\to\infty$ as $n\to\infty$,  and therefore the limit distribution is Fr\'echet.

\item[b)] Now, let $\ell(\cdot)$ be such that $\ell(u)\to\tilde{c} > 0$ as $u\to\infty$. Then the same choice of norming constants as when $\ell(u)\to\infty$ as $u\to\infty$ leads to the same limits as before. However, the value of~\eqref{lim const} now depends on $x$, namely,
\begin{equation*}
\lim\limits_{n\to\infty} k_n ^{1-\lambda (c \ell(s_n))^{\nicefrac{\tau}{\alpha}} x^{\tau}} =
\begin{cases}
\infty,& x\in(0, \lambda^{-\nicefrac{1}{\tau}} (c\tilde{c})^{-\nicefrac{1}{\alpha}}),\\
1,& x=\lambda^{-\nicefrac{1}{\tau}} (c\tilde{c})^{-\nicefrac{1}{\alpha}},\\
0,& x>\lambda^{-\nicefrac{1}{\tau}} (c\tilde{c})^{-\nicefrac{1}{\alpha}}.
\end{cases}
\end{equation*}

Thus, in this case the limit distribution is equal to \begin{equation*}
\lim\limits_{n\to\infty} \P \lb\max\limits_{j=1,\dots, k_n} X_{nj} \leq v_n (x)\rb = 
\begin{cases}
0,& x< \lambda^{-\nicefrac{1}{\tau}} (c\tilde{c})^{-\nicefrac{1}{\alpha}},\\
e^{-(1+x^{-\alpha})},& x=\lambda^{-\nicefrac{1}{\tau}} (c\tilde{c})^{-\nicefrac{1}{\alpha}},\\
e^{-x^{-\alpha}},& x>\lambda^{-\nicefrac{1}{\tau}} (c\tilde{c})^{-\nicefrac{1}{\alpha}}.
\end{cases}
\end{equation*}
An interesting point is that we get the limit distribution that is not from the extreme value family, having an atom at $x=\lambda^{-\nicefrac{1}{\tau}} (c\tilde{c})^{-\nicefrac{1}{\alpha}}$.

\item[c)] Finally, let $\ell(\cdot)$ be such that $\ell(u)\to 0$ as $u\to\infty$. Then the normalising sequence can be chosen as in item 1, and for the first component we get
\begin{equation*}
\lim\limits_{n\to\infty} k_n e^{-\lambda v_n ^{\tau} (x)} = e^{-x} \quad \forall x\in\R,
\end{equation*}
while for the second one
\begin{multline*}
\lim\limits_{n\to\infty} \frac{k_n \eps_n \ell(c_n)}{c_n ^{\alpha}} = \lim\limits_{n\to\infty} \frac{c\cdot (\log k_n)^{\nicefrac{\alpha}{\tau}} \ell\left((\log k_n)^{\nicefrac{1}{\tau}}\right)}{(\log k_n)^{\nicefrac{\alpha}{\tau}}} \\= \lim\limits_{n\to\infty} c\cdot \ell\left((\log k_n)^{\nicefrac{1}{\tau}}\right) = 0.
\end{multline*}
Therefore, in this case the limit distribution is again Gumbel.
\end{itemize}
\section{Limit law for the truncated RV distribution}
\label{app3} 
\begin{lemma}
Let $\widetilde{F}_2  (x)=F_2(x; m,M,L)$ be the upper-truncated regularly varying distribution defined as~\eqref{truncated RV}. Then $\widetilde{F}_2$ is in the maximum domain of attraction of the Weibull law \(\Psi_\tau\) having the distribution function~\eqref{Weib} with \(\tau=1\).\end{lemma}

\begin{proof}
As it is known, $F\in MDA(\Psi_{\tau})$ for some $\tau>0$ if and only if $x^* = \sup\lb x\in\R: F(x) <1\rb<\infty$ and $\bar{F}\left(x^* - \frac{1}{x}\right)\in RV_{-\tau}$, see, e.g, \cite{EKM}. Thus, $\widetilde{F}_2 \in MDA(\Psi_{\tau})$ for some $\tau>0$ if and only if
\begin{equation*}
\bar{F} \left(M-\frac{1}{x}\right) = x^{-\tau} \tilde{\ell} (x), \quad \tau>0,
\end{equation*}
for some slowly varying function $\tilde{\ell}(\cdot)$, or, equivalently, iff
\begin{equation*}
G(x) := x^{\tau} \bar{F} \left(M-\frac{1}{x}\right)\in RV_0, \quad \tau>0.
\end{equation*}
Therefore, to prove this statement of this lemma , we need to show that \(\lim_{x\to\infty} (G(tx)/G(x))=1\) for any \(t>0\), that is, 
\begin{multline}
\label{MDA}
\lim\limits_{x\to\infty} \frac{G(tx)}{G(x)} = \lim\limits_{x\to\infty} t^{\tau}\frac{\bar{F}\left(M - \frac{1}{tx}\right)}{\bar{F}\left(M - \frac{1}{x}\right)}  \\= \lim\limits_{x\to\infty} t^{\tau} \frac{\left(M - \frac{1}{tx}\right)^{-\alpha} \ell\left(M-\frac{1}{tx}\right) - M^{-\alpha} \ell(M)}{\left(M - \frac{1}{x}\right)^{-\alpha} \ell\left(M-\frac{1}{x}\right) - M^{-\alpha} \ell(M)} =1 \quad \forall t>0.
\end{multline}
First, 
\begin{equation*}
\lim\limits_{x\to\infty} \left(M-\frac{1}{x}\right)^{-\alpha} = \lim\limits_{x\to\infty} M^{-\alpha} \left(1-\frac{1}{Mx}\right)^{-\alpha}=M^{-\alpha}+\alpha M^{-\alpha-1}\cdot\frac{1}{x} (1+\bar{o}(1)).
\end{equation*}
Then, assuming that $\ell(\cdot)$ is continuous and differentiable,
\begin{equation*}
\lim\limits_{x\to\infty} \ell\left(M-\frac{1}{x}\right) = \ell(M) - \ell'(M) \cdot \frac{1}{x} (1+\bar{o}(1)).
\end{equation*}
Therefore,
\begin{multline*}
\lim\limits_{x\to\infty} t^{\tau} \frac{\left(M - \frac{1}{tx}\right)^{-\alpha} \ell\left(M-\frac{1}{tx}\right) - M^{-\alpha} \ell(M)}{\left(M - \frac{1}{x}\right)^{-\alpha} \ell\left(M-\frac{1}{x}\right) - M^{-\alpha} \ell(M)}\\= \lim\limits_{x\to\infty} t^{\tau}\frac{\left(M^{-\alpha}+\frac{\alpha M^{-\alpha-1}}{tx} (1+\bar{o}(1))\right)\left(\ell(M) - \frac{\ell'(M)}{tx} (1+\bar{o}(1))\right) - M^{-\alpha} \ell(M)}{\left(M^{-\alpha}+\frac{\alpha M^{-\alpha-1}}{x} (1+\bar{o}(1))\right)\left(\ell(M) - \frac{\ell'(M)} {x} (1+\bar{o}(1))\right) - M^{-\alpha} \ell(M)} \\=\lim\limits_{x\to\infty} t^{\tau}\frac{\left(\frac{M^{-\alpha} \ell' (M)}{tx} + \frac{\alpha M^{-\alpha-1}}{tx} \ell(M)\right)(1+\bar{o}(1))}{\left(\frac{M^{-\alpha} \ell' (M)}{x} + \frac{\alpha M^{-\alpha-1}}{x} \ell(M)\right)(1+\bar{o}(1))}.
\end{multline*}
Clearly, the latter limit is equal to one if $\tau=1$, meaning that $\widetilde{F}_2 \in MDA(\Psi_1)$. 
\end{proof}
\section{Proof of Theorem~\ref{thm2}}\label{app4}
\textbf{Step 1. Several simple cases.} 
As in the proof of Theorem~\ref{thm1}, we use the notation \(v_n(x) =s_n x +c_n\). We have 
\begin{eqnarray*}
H(x) &=&
\lim\limits_{n\to\infty} \Biggl((1-\eps_n)(1-e^{-\lambda v_n ^{\tau} (x)}) \Biggl. \\ && 
\hspace{2cm} \Biggl.+ \eps_n \left(\frac{1-v_n ^{-\alpha} (x) \ell(v_n (x))}{1-M_n ^{-\alpha} \ell(M_n)} \1_{\{v_n (x) \in[m, M_n]\}} + \1_{\{v_n(x) > M_n\}}\right)\Biggl)^{k_n}\\ 
&=&
\lim\limits_{n\to\infty} \Biggl(1-(1-\eps_n)e^{-\lambda v_n ^{\tau} (x)} \Biggl.\\&&\hspace{1cm} \Biggl.+ \eps_n \left(\frac{M_n ^{-\alpha} \ell(M_n) -v_n ^{-\alpha} (x) \ell(v_n (x))}{1-M_n ^{-\alpha} \ell(M_n)} \1_{\{v_n (x) \in[m, M_n]\}} - \1_{\{v_n(x) < m\}}\right)\Biggl)^{k_n},
\end{eqnarray*}
where we use that  $\1_{\{v_n (x) >M_n\}} = 1 - \1_{\{v_n (x) \in [m, M_n]\}} - \1_{\{v_n (x)< m\}}$.
Since $\eps_n\to 0$, $M_n\to\infty$ and $M_n ^{-\alpha} \ell(M_n) \to 0$ as $n\to\infty$ by slow variation of $\ell(\cdot)$, we get 
\begin{eqnarray*}
H(x) &=& 
\lim\limits_{n\to\infty} \left(1 - e^{-\lambda v_n ^{\tau} (x)} \right. \\ 
&&
\left. \hspace{1cm}+ \eps_n ((M_n^{-\alpha} \ell(M_n) - v_n ^{-\alpha} (x) \ell(v_n (x)))\1_{\{v_n (x)\in[m, M_n]\}} - \1_{\{v_n (x) < m\}})\right)^{k_n}\\
&=& 
\exp \Bigl\{ 
- \lim_{n\to\infty} \Bigl(
k_n e^{-\lambda v_n ^{\tau} (x)} \\
&& \hspace{1cm}- k_n \eps_n (M_n ^{-\alpha} \ell(M_n) - v_n ^{-\alpha} (x) \ell(v_n (x)))\1_{\{v_n (x)\in[m, M_n]\}}+k_n\eps_n\1_{\{v_n (x) <m\}}\Bigr)
\Bigr\}.
\end{eqnarray*}

\begin{itemize}
\item[(\textit{i})] Let $k_n \eps_n\to 0$ as $n\to\infty$. By slow variation of $\ell(\cdot)$,  $v_n ^{-\alpha} (x) \ell(v_n (x))\to 0$ as $n\to\infty$, therefore, the whole second component disappears. We deal with maxima of a Weibull random variable and obtain the Gumbel limit under the normalisation $v_n (x) = s_n x + c_n$ with \(s_n, c_n\) in the form~\eqref{cond1}.

\item[(\textit{ii})] Let $k_n \eps_n = const$. By a similar argument as in the previous item, 
\begin{equation*}
H(x) =\lim\limits_{n\to\infty} \left(1 - \frac{1}{k_n} [k_n e^{-\lambda v_n ^{\tau} (x)} + k_n \eps_n \1_{\{v_n (x) < m\}}]\right)^{k_n}.
\end{equation*}
Under the same choice of normalising sequences, we get 
\begin{equation*}
v_n (x) = \frac{x}{\lambda\tau} \left(\frac{\log k_n}{\lambda}\right)^{\nicefrac{1}{\tau}-1}  + \left(\frac{\log k_n}{\lambda}\right)^{\nicefrac{1}{\tau}} = \left(\frac{\log k_n}{\lambda}\right)^{\nicefrac{1}{\tau}}\left(\frac{x}{\tau\log k_n} + 1\right)\to\infty 
\end{equation*}
for all fixed $x\in\R$ as $n\to\infty$.
Therefore, the limit distribution is again Gumbel.

\item[(\textit{iii})] Let $k_n \eps_n\to\infty$ as $n\to\infty$. The further analysis depends on the asymptotic properties of \(M_n.\) If~\eqref{M1} holds,   the proof is based on the observation that 
\begin{equation*}
M_n ^{-\alpha} \ell(M_n) \lesssim M_n^{-\alpha+\epsilon}\quad\forall\epsilon>0,
\end{equation*}
and therefore,
\begin{equation*}
k_n \eps_n M_n ^{-\alpha} \ell(M_n) \lesssim k_n \eps_n M_n^{-\alpha+\epsilon}\quad\forall\epsilon>0.
\end{equation*}
From~\eqref{M1} it follow that for any $\epsilon\in(0, \alpha -\nicefrac{1}{\gamma}]$ the right-hand side tends to zero as $n\to\infty$, and therefore $k_n\eps_n M_n ^{-\alpha} \ell(M_n) \to 0$ as  $n \to \infty$. The rest of the proof in this situation follows the same lines as the proof of Theorem~\ref{thm1}. Other cases are more complicated, and we divide the further proof  into several steps.
\end{itemize}
\textbf{Step 2. Case~\eqref{M2}.}  Recalling again that $\ell(\cdot)\in RV_0$, we get that
\begin{equation*}
k_n \eps_n M_n ^{-\alpha} \ell(M_n) \gtrsim k_n \eps_n M_n^{-\alpha-\epsilon}\quad\forall\epsilon>0.
\end{equation*}
Therefore, for any $\epsilon\in(0, \nicefrac{1}{\gamma}-\alpha],$ $k_n \eps_n M_n ^{-\alpha} \ell(M_n)\to\infty$ as $n\to\infty$. Now, assume that $v_n (x)$ is such that $v_n (x)\leq M_n$ for all $x\in\R$ and \(n\) large enough. Then 
\begin{multline*}
\lim_{x \to \infty} H(x)\\ = 
 \exp \Bigl\{ 
- \lim_{n\to\infty}\lim\limits_{x\to\infty} \Bigl(
k_n e^{-\lambda v_n ^{\tau} (x)} 
- k_n \eps_n (M_n ^{-\alpha} \ell(M_n) - v_n ^{-\alpha} (x) \ell(v_n (x)))\1_{\{v_n (x)\in[m, M_n]\}}\\+ k_n\eps_n\1_{\{v_n (x) <m\}}\Bigr)
\Bigr\}= e^{\lim\limits_{n\to\infty} k_n \eps_n M_n ^{-\alpha}\ell(M_n)} = \infty,
\end{multline*}
and therefore the limit distribution doesn't exist. We conclude that there is a non-degenerate limit distribution only if $v_n (x) >M_n$ for all $x \in \R$ and \(n\) large enough. 
In this case, 
\begin{equation*}
k_n e^{-\lambda v_n ^{\tau} (x)} = - \log(H(x)) (1+\bar{o}(1)).
\end{equation*}
The condition $v_n (x)>M_n$ leads to the inequality 
\begin{equation}\label{MM}
\frac{k_n}{e^{\lambda M_n ^{\tau}}} > - \log(H(x))(1+\bar{o}(1))
\end{equation}
for $n$ sufficiently large. Finally, as $-\log(H(x))$ takes only non-negative values, and the sequences \(k_n, M_n\) tend to infinity as \(n\to \infty,\) we conclude that the necessary condition for the existence of non-degenerate limit distribution is 
\begin{equation}
\label{Mn2}
M_n ^{\tau} \lesssim \log k_n.
\end{equation}

\begin{enumerate}
\item Now let us consider the case when~\eqref{A1} holds. 
We have \begin{equation*}
\frac{M_n^\tau}{\log k_n } \lesssim \frac{M_n ^{\tau}}{(k_n \eps_n)^{\beta}}\lesssim\left( \frac{M_n}{(k_n \eps_n)^{\gamma}}\right)^\tau,
\end{equation*}
since \(\beta > \nicefrac{\tau}{\alpha}>\tau \gamma$. The right-hand side tends to zero as $n\to\infty$ and therefore~\eqref{Mn2} holds. Choosing $s_n$ and $c_n$ as in~\eqref{const1}, we get that~$v_n (x)>M_n$ for all \(x \in \R\), and therefore the limit distribution  is Gumbel.

\item Now assume that~\eqref{A2} holds. In this case,~\eqref{Mn2} can be violated. In fact, 
\begin{equation*}
\frac{M_n ^{\tau}}{\log k_n} = \left(\frac{ M_n}{(\log k_n)^{\nicefrac{1}{\tau}}}\right)^{\tau} \lesssim \left(\frac{ (k_n \eps_n)^{\gamma}}{(\log k_n)^{\nicefrac{1}{\tau}}}\right)^{\tau} = \frac{(k_n \eps_n)^{\tau\gamma}}{\log k_n}
\end{equation*}
with some $\gamma\in(0, \nicefrac{1}{\alpha})$. From~\eqref{M2}, it follows that  the right-hand side is infinite if $\tau\gamma\geq\beta$, and has an unknown asymptotic behaviour otherwise. The lower bound is given by
\begin{equation*}
\frac{ M_n ^{\tau}}{\log k_n} \gtrsim \frac{M_n ^{\tau}}{(k_n \eps_n)^{\beta}} = \left(\frac{M_n}{(k_n \eps_n)^{\nicefrac{\beta}{\tau}}}\right)^{\tau},
\end{equation*}
where for $\beta\geq\tau\gamma$, the right-hand side tends to zero as $n\to\infty$, while otherwise the asymptotic behaviour is again unknown. 
In this case, we conclude that if  $M_n$ is such that~\eqref{Mn2} holds, the non-degenerate limit distribution exists and is in fact the Gumbel distribution. 

\item Finally, in the case~\eqref{A3},
\begin{equation*}
\frac{ M_n ^{\tau}}{\log k_n} = \left(\frac{ c^{\nicefrac{1}{\alpha}} M_n}{(k_n \eps_n)^{\nicefrac{1}{\alpha}}}\right)^{\tau} \lesssim \left(\frac{ c^{\nicefrac{1}{\alpha}} M_n}{(k_n \eps_n)^{\gamma}}\right)^{\tau},
\end{equation*}
because $\gamma\in(0, \nicefrac{1}{\alpha})$. Since by~\eqref{Mn2} there exists $\gamma\in(0, \nicefrac{1}{\alpha})$ such that the right-hand side tends to zero as $n\to\infty$, we get under proper normalisation the Gumbel limit distribution.
\end{enumerate}

\textbf{Step 3. Case~\eqref{M3}.}
\begin{enumerate}
\item If~\eqref{A1} is satisfied, it is possible to obtain the Gumbel limit under the same choice of normalising sequence~\eqref{const1}. Indeed, in this case
\begin{equation*}
\frac{ M_n^\tau}{\log k_n} = \frac{\breve{c}^\tau (k_n \eps_n)^{\nicefrac{\tau}{\alpha}}}{\log k_n }
\lesssim \frac{(k_n \eps_n)^{\beta}}{\log k_n},
\end{equation*}
because $\beta>\nicefrac{\tau}{\alpha}$. Therefore, \eqref{Mn2} follows from~\eqref{A1}, and  we obtain the Gumbel distribution as a limit. 
\item If \eqref{A2} holds, then the result turns out to depend on the asymptotic behaviour of $\ell(\cdot)$. Let us recall that \(k_n \eps_n M_n ^{-\alpha}\) is equal to a constant.
\begin{itemize}
\item[a)] If $\ell(u)\to\infty$ as $u\to\infty$, we have that $k_n \eps_n M_n ^{-\alpha} \ell(M_n)\to\infty$ as $n\to\infty$. Thus, as was argued above, the non-degerated limit \(H(x)\) exists if and only if~\eqref{Mn2} holds for all $x$ and \(n\) large enough. However, for any  $\beta \in (0, \nicefrac{\tau}{\alpha})$,
\begin{equation*}
\frac{M_n^\tau}{\log k_n} = \frac{\breve{c}^\tau (k_n \eps_n)^{\nicefrac{\tau}{\alpha}}}{\log k_n }\gtrsim \frac{(k_n \eps_n)^{\beta}}{\log k_n},
\end{equation*}
and therefore the assumption \eqref{Mn2} is violated due to~\eqref{A2}. Thus, in this case there exists no non-degenerate limit distribution.

\item[b)] Now consider the case $\ell(u)\to\tilde{c}$ for some $\tilde{c}>0$ as $u\to\infty$. Let us fix \(s_n, c_n\) in the form~\eqref{const2}. The inequality $v_n(x) = s_n x > M_n$ is equivalent to \(x > \breve{c}\tilde{c}^{-\nicefrac{1}{\alpha}}.\)
Under this normalisation, we have
\begin{equation*}
\lim\limits_{n\to\infty} k_n e^{-\lambda v_n ^{\tau} (x)} = 0,
\end{equation*}
and therefore 
\[
H(x) =\begin{cases} e^{\tilde{c}\breve{c}^{-\alpha} - x^{-\alpha}},& x\in(0, \breve{c}\tilde{c}^{-\nicefrac{1}{\alpha}}],\\ 1,& x>\breve{c}\tilde{c}^{-\nicefrac{1}{\alpha}}. \end{cases}
\]

\item[c)] If $\ell(u)\to 0$ as $u\to\infty$, one can take the norming constants as in the previous item and obtain the Fr\'echet limit distribution since $k_n \eps_n M_n ^{-\alpha} \ell(M_n)\to 0$ and $k_n \eps_n s_n ^{-\alpha} \ell(s_n) x^{-\alpha} \to x^{-\alpha}$ as $n\to\infty$. The last thing which is crucial here is to check that $v_n(x) \leq M_n$ for all $x\in\R$. This inequality follows from 
\begin{equation*}
s_nx = (k_n \eps_n)^{\nicefrac{1}{\alpha}} \ell^{\nicefrac{1}{\alpha}} (s_n)x = \breve{c}^{-1} M_n \ell^{\nicefrac{1}{\alpha}} (s_n)x \lesssim M_n.
\end{equation*}
\end{itemize}

\item Finally, let us consider the case~\eqref{A3}. As in the previous situations, the limit distribution depends on the asymptotic behaviour of $\ell(\cdot)$. 
\begin{itemize}
\item[a)] Let $\ell(u)\to\infty$ as $u\to\infty$. Since then $k_n \eps_n M_n ^{-\alpha} \ell(M_n) \to \infty$ as $n\to\infty$, we conclude that the non-degenerate limit exists only if~\eqref{MM} holds. In the considered case, \eqref{MM} is equivalent to 
\begin{equation*}
\lambda^{1/\tau} \breve{c}c^{1/\alpha}<1
\end{equation*}
The normalising sequence~\eqref{const1} again leads to the Gumbel limit distribution.

\item[b)] If $\ell(u)\to\tilde{c}$ for some $\tilde{c}>0$ as $u\to\infty$, we have that $k_n \eps_n M_n ^{-\alpha} \ell(M_n) = \tilde{c}\breve{c}^{-\alpha}$. 

\item If $\lambda^{\nicefrac{1}{\tau}} \breve{c} c^{\nicefrac{1}{\alpha}} < 1$, a linear normalising sequence as in~\eqref{const1} leads to the Gumbel limit, since $v_n (x)>M_n$ for all $x\in\R$ and $n$ large enough; see the previous item.

\item If $\lambda^{\nicefrac{1}{\tau}} \breve{c} c^{\nicefrac{1}{\alpha}} > 1$, the choice~\eqref{const2}  of normalising constants yields $s_n x >M_n$ for all $x>\breve{c}\tilde{c}^{-\nicefrac{1}{\alpha}}$, while
\begin{equation*}
\lim\limits_{n\to\infty} k_n e^{-\lambda s_n ^{\tau} x^{\tau}} = k_n ^{1-(c\tilde{c})^{\nicefrac{\tau}{\alpha}}x^{\tau}} = 
\begin{cases}
\infty,& x\in(0, \lambda^{-\nicefrac{1}{\tau}} (c\tilde{c})^{-\nicefrac{1}{\alpha}}),\\
1,& x=\lambda^{-\nicefrac{1}{\tau}} (c\tilde{c})^{-\nicefrac{1}{\alpha}},\\
0,& x>\lambda^{-\nicefrac{1}{\tau}} (c\tilde{c})^{-\nicefrac{1}{\alpha}}
\end{cases}
\end{equation*}
and 
\begin{equation*}
\lim\limits_{n\to\infty} k_n \eps_n s_n ^{-\alpha} \ell(s_n) x^{-\alpha} = x^{-\alpha}.
\end{equation*}
Therefore, we get 
\begin{eqnarray*}
H(x)=
\begin{cases}
0,& x\in(0, \lambda^{-\nicefrac{1}{\tau}}(c\tilde{c})^{-\nicefrac{1}{\alpha}}),\\
\exp\lb -1+\tilde{c}\breve{c}^{-\alpha} - \lambda^{\nicefrac{\alpha}{\tau}} c\tilde{c} \rb,& x=\lambda^{-\nicefrac{1}{\tau}}(c\tilde{c})^{-\nicefrac{1}{\alpha}},\\
\exp\lb \tilde{c}\breve{c}^{-\alpha}-x^{-\alpha} \rb,& x\in(\lambda^{-\nicefrac{1}{\tau}}(c\tilde{c})^{-\nicefrac{1}{\alpha}}, \breve{c}\tilde{c}^{-\nicefrac{1}{\alpha}}],\\
1,& x> \breve{c}\tilde{c}^{-\nicefrac{1}{\alpha}}.
\end{cases}
\end{eqnarray*}
As we see, the limit distribution does not belong to the extreme value family, and has an atom at $x=\lambda^{-\nicefrac{1}{\tau}}(c\tilde{c})^{-\nicefrac{1}{\alpha}}$.
\item  If $\lambda^{\nicefrac{1}{\tau}} \breve{c} c^{\nicefrac{1}{\alpha}} = 1$, the choice~\eqref{const2} leads to the discrete limit distribution having a unique atom at \(x=\breve{c}\tilde{c}^{-\nicefrac{1}{\alpha}}\) with probability mass \(1/e.\)

\item[c)] Lastly, let $\ell(\cdot)$ be such that $\ell(u)\to 0$ as $u\to\infty$. Then the Gumbel limit can be obtained under the normalisation~\eqref{const1} since $k_n \eps_n M_n ^{-\alpha} \ell(M_n)\to 0$ as $n\to\infty$, see item (i,c) in Theorem~\ref{thm1}. This observation completes the proof.
\end{itemize}
\end{enumerate}
\bibliographystyle{spbasic}
\bibliography{Panov_bibliography}
\end{document}